\documentclass[11pt,reqno]{amsart}

\usepackage[T1]{fontenc}
\usepackage[utf8]{inputenc}
\usepackage{lmodern}
\usepackage{amsmath,amssymb,amsthm,mathrsfs, amsfonts}
\usepackage[hmargin=1in,vmargin=0.7in]{geometry}
\usepackage{xcolor}
\usepackage{booktabs}
\usepackage{colortbl}
\usepackage{enumitem}
\usepackage{mathtools}
\mathtoolsset{showonlyrefs}
\definecolor{tabhead}{RGB}{222, 231, 243}
\definecolor{tabrow}{RGB}{244, 247, 251}
\definecolor{lava}{rgb}{0.81, 0.06, 0.13}
\definecolor{cobalt}{rgb}{0.0, 0.28, 0.67}

\usepackage[colorlinks=true,linkcolor=lava,citecolor=blue,urlcolor=cobalt,hypertexnames=false]{hyperref}

\theoremstyle{plain}
\newtheorem{theorem}{Theorem}[section]
\newtheorem{proposition}{Proposition}[section]
\newtheorem{lemma}{Lemma}[section]

\theoremstyle{definition}
\newtheorem{definition}{Definition}[section]
\newtheorem{remark}{Remark}[section]
\newtheorem{question}{Question}[section]
\newtheorem{assumption}{Assumption}[section]
\numberwithin{equation}{section}

\newcommand{\R}{\mathbb{R}}
\newcommand{\N}{\mathbb{N}}
\newcommand{\C}{\mathbb{C}}
\newcommand{\Lc}{\mathcal{L}_a}
\newcommand{\Ha}{H^1_a}
\newcommand{\dHa}{\dot H^1_a}
\newcommand{\Ea}{E_a}
\newcommand{\Qa}{\mathcal{Q}_a}
\newcommand{\Sw}{S_\omega}
\newcommand{\GN}{\mathcal{C}_a}
\newcommand{\GNs}{\mathcal{C}_{\sharp}}
\newcommand{\AH}{\mathcal{A}_{\mathrm{H}}}
\newcommand{\aH}{a_{\mathrm{H}}}
\newcommand{\Mass}{\mathcal{M}}
\newcommand{\eps}{\varepsilon}
\newcommand{\dd}{\,dx}
\DeclareMathOperator{\Imm}{Im}
\DeclareMathOperator{\Ree}{Re}

\begin{document}

\title[Normalized ground states and a mass-constrained threshold]
{Normalized ground states as sharp dynamical thresholds for the inhomogeneous NLS with an inverse-square potential}

\author[M. Majdoub]{Mohamed Majdoub}
\address[M. Majdoub]{Department of Mathematics, College of Science,
Imam Abdulrahman Bin Faisal University, P.O.~Box 1982, Dammam 31441, Saudi Arabia}
\address[M. Majdoub]{Basic and Applied Scientific Research Center,
Imam Abdulrahman Bin Faisal University, P.O.~Box 1982, Dammam 31441, Saudi Arabia}
\email{mmajdoub@iau.edu.sa}
\urladdr{https://orcid.org/0000-0001-6038-1069}

\author[T. Saanouni]{Tarek Saanouni}
\address[T. Saanouni]{Department of Mathematics, College of Science,
Qassim University, Buraydah 51452, Saudi Arabia}
\email{t.saanouni@qu.edu.sa}
\urladdr{https://orcid.org/0000-0003-1433-3650}

\subjclass[2020]{Primary 35Q55; Secondary 35J61, 35A15, 35B44, 35B35, 35P25, 35A23}
\keywords{Inhomogeneous nonlinear Schr\"odinger equation; inverse-square potential;
normalized ground states; Pohozaev manifold; sharp Gagliardo--Nirenberg inequality;
scattering; finite-time blow-up; strong instability; virial identity}

\begin{abstract}
We study the focusing inhomogeneous nonlinear Schr\"odinger equation with an
inverse-square potential, in the mass-supercritical and energy-subcritical regime
and in every dimension $N\ge3$. Two lines of research on this model have so far
developed separately, namely the scattering theory below the ground state and the
variational theory of solutions with prescribed mass. We connect them. We compute
explicitly the minimax level $m_a(c)$ of the energy on the Pohozaev manifold of
mass $c$ in terms of the sharp constant of
the associated Gagliardo--Nirenberg inequality. We prove that the minimizers of
$m_a(c)$ are exactly the optimizers of this inequality lying on the Pohozaev
manifold, and that they are ground states of the stationary problem. This answers
a question left open in the literature. We then show that $m_a(c)$ is a sharp
dynamical threshold. Below it, the sign of the Pohozaev functional is preserved by
the flow. If this sign is positive, the solution is global. If it is negative and
the datum has finite variance, the solution blows up in finite time. For the
three-dimensional cubic model, the global solutions also scatter, by a known
scattering criterion read in its scale-invariant form. As
consequences, we obtain the strong instability of all normalized ground states,
with destabilizing data on the same mass sphere, the strict monotonicity of the
threshold with respect to the coupling constant, and a positive lower bound for
$m_a(c)$ which is uniform up to the critical Hardy coupling.
\end{abstract}

\maketitle

%%%%%%%%%%%%%%%%%%%%%%%%%%%%%%%%%%%%%%%%%%%%%%%%%%%%%%%%%%%%%%%%%%%%%%%%%%%%%%
\section{Introduction}\label{sec:intro}
%%%%%%%%%%%%%%%%%%%%%%%%%%%%%%%%%%%%%%%%%%%%%%%%%%%%%%%%%%%%%%%%%%%%%%%%%%%%%%

This paper deals with the Cauchy problem for the focusing inhomogeneous nonlinear
Schr\"odinger equation with an inverse-square potential,
\begin{equation}\label{eq:main}
\begin{cases}
i\partial_t u-\Lc u+|x|^{-b}|u|^{\alpha}u=0,&(t,x)\in\R\times\R^{N},\\
u(0,\cdot)=u_0,
\end{cases}\tag{INLS}
\end{equation}
where $N\ge3$, $u:\R\times\R^{N}\to\C$, $0<b<2$, $\alpha>0$, and
\begin{equation}\label{eq:La}
\Lc:=-\Delta+\frac{a}{|x|^{2}},\qquad a>\aH:=-\frac{(N-2)^{2}}{4}.
\end{equation}
The lower bound on $a$ comes from the sharp Hardy inequality (see, for instance,
\cite{HLP1952})
\begin{equation}\label{eq:hardy}
\frac{(N-2)^{2}}{4}\int_{\R^{N}}\frac{|u|^{2}}{|x|^{2}}\dd
\le\int_{\R^{N}}|\nabla u|^{2}\dd,\qquad u\in H^{1}(\R^{N}).
\end{equation}
For $a>\aH$, this inequality implies that the quadratic form
\begin{equation}\label{eq:dHa}
\|u\|_{\dHa}^{2}:=\big\|\Lc^{1/2}u\big\|_{L^{2}}^{2}
=\int_{\R^{N}}\Big(|\nabla u|^{2}+a\frac{|u|^{2}}{|x|^{2}}\Big)\dd
\end{equation}
is comparable to $\|\nabla u\|_{L^2}^2$. We take for $\Lc$ the Friedrichs
extension of the operator defined on $C_c^{\infty}(\R^{N}\setminus\{0\})$. This
choice matters for $\aH<a<\aH+1$, where this operator has several self-adjoint
extensions \cite{KSWW1975}. We write
$\Ha(\R^{N})$ for $H^{1}(\R^{N})$ endowed with the norm
$\|u\|_{\Ha}^{2}=\|u\|_{L^{2}}^{2}+\|u\|_{\dHa}^{2}$.

Both perturbations of the Laplacian in \eqref{eq:main} are physically natural,
and each of them alone has been studied in depth. The potential $a|x|^{-2}$ has
the same homogeneity as $-\Delta$, and consequently it is invariant under the
natural scaling of the equation. Therefore it cannot be regarded as a
short-range perturbation, and it does not vanish in any asymptotic regime
generated by the scaling flow. Operators of this type occur in the quantum
mechanics of a charged particle near a polar molecule, in the linearization of critical elliptic problems, in combustion
models, and in wave propagation on conical geometries. We refer to Kalf,
Schmincke, Walter and W\"ust \cite{KSWW1975}, Burq, Planchon, Stalker and Tahvildar-Zadeh
\cite{BPST2003} and Killip, Miao, Visan, Zhang and Zheng \cite{KMVZZ2018}. The
weight $|x|^{-b}$ describes a medium whose nonlinear response varies in space. In
nonlinear optics, it models the propagation of a laser beam in a plasma channel
with nonuniform density; see Gill \cite{Gill2000}. It also appears in the study
of Bose--Einstein condensates with a spatially modulated scattering length. In \eqref{eq:main}, the
two singularities sit at the same point. The potential acts on the linear flow at
the scale-invariant rate, whereas the weight changes the strength of the
nonlinearity near the origin. Much of the interest of the model lies in the
interaction of these two effects.

Equation \eqref{eq:main} is invariant under the scaling
\begin{equation}\label{eq:scaling}
u_\mu(t,x):=\mu^{\frac{2-b}{\alpha}}u(\mu^{2}t,\mu x),\qquad\mu>0.
\end{equation}
This scaling leaves the homogeneous Sobolev space $\dot H^{s_c}(\R^N)$ invariant,
with
\begin{equation}\label{eq:sc}
s_c:=\frac N2-\frac{2-b}{\alpha}.
\end{equation}
Solutions of \eqref{eq:main} conserve the mass and the energy,
\begin{equation}\label{eq:conserved}
\Mass(u):=\int_{\R^{N}}|u|^{2}\dd,\qquad
\Ea(u):=\frac12\|u\|_{\dHa}^{2}-\frac{1}{\alpha+2}\int_{\R^{N}}\frac{|u|^{\alpha+2}}{|x|^{b}}\dd .
\end{equation}
Throughout the paper, we work in the mass-supercritical and energy-subcritical regime
\begin{equation}\label{eq:range}
\frac{4-2b}{N}<\alpha<\frac{4-2b}{N-2},
\qquad\text{equivalently}\qquad 0<s_c<1.
\end{equation}
In this range, the energy is unbounded from below on every sphere of $L^{2}$. The
long-time behaviour of solutions then results from a competition between
dispersion and concentration.

\subsection{Background}\label{ssec:background}

\emph{The classical equation.} When $a=0$ and $b=0$, \eqref{eq:main} is the usual
focusing NLS, and its threshold theory in the range \eqref{eq:range} is well
understood. Weinstein \cite{Weinstein1983} identified the sharp constant in the
Gagliardo--Nirenberg inequality in terms of the ground state $Q$, which led to the
first sharp criteria for global existence. Berestycki and Cazenave
\cite{BerestyckiCazenave1981} showed that standing waves are strongly unstable in
the mass-supercritical case, using the virial identity of Glassey
\cite{Glassey1977}. Holmer and
Roudenko \cite{HolmerRoudenko2008} then proved that, for the three-dimensional
cubic equation, radial data satisfying
\begin{equation}\label{eq:HR}
\Mass(u_0)\Ea(u_0)<\Mass(Q)\Ea(Q),\qquad
\|u_0\|_{L^{2}}\|\nabla u_0\|_{L^{2}}<\|Q\|_{L^{2}}\|\nabla Q\|_{L^{2}},
\end{equation}
lead to global solutions that scatter. If the second inequality is reversed, then
finite-variance data lead to blow-up (see \cite{HolmerRoudenko2008} and the
references therein). The proof of scattering followed the concentration-compactness
and rigidity scheme of Kenig and Merle \cite{KenigMerle2006}. The radial
assumption was removed by Duyckaerts, Holmer and Roudenko \cite{DHR2008}, and the
whole intercritical range was treated by Fang, Xie and Cazenave \cite{FXC2011}.
Duyckaerts and Roudenko \cite{DuyckaertsRoudenko2010} classified the solutions
lying exactly on the threshold. More recently, Dodson and Murphy
\cite{DodsonMurphy2017} found a proof of scattering that avoids
concentration-compactness. It combines a virial--Morawetz estimate with a
scattering criterion. The same circle of ideas was used for NLS with exponential
nonlinearities in two dimensions \cite{KeraaniMajdoub2020,BDM2019}.

\emph{The inhomogeneous equation.} For $a=0$ and $b>0$, the local theory in
$H^{1}$ goes back to Genoud and Stuart \cite{GenoudStuart2008} and Guzm\'an
\cite{Guzman2017}. Farah \cite{Farah2016} proved the sharp
Gagliardo--Nirenberg inequality with weight $|x|^{-b}$ and the corresponding
global existence and blow-up results. Scattering below the ground state was first
obtained by Farah and Guzm\'an \cite{FarahGuzman2017} for radial data in three
dimensions. The radial assumption was removed by Miao, Murphy and Zheng
\cite{MiaoMurphyZheng2021} and by Campos and Cardoso \cite{CamposCardoso2022}.
The decay of the weight at infinity was exploited by Dinh and Keraani
\cite{DinhKeraani2021} to obtain blow-up criteria for non-radial data of possibly
infinite variance. The strong instability of the ground states was proved by
Ardila and Cardoso \cite{ArdilaCardoso2021}, and the energy-critical case was
treated by Guzm\'an and Murphy \cite{GuzmanMurphy2021}. Related threshold results for weights that grow at
infinity, and for competing inhomogeneous nonlinearities where scaling is lost,
were obtained in \cite{DMS2023,GMS2026}.

\emph{The inverse-square potential.} For $a\ne0$ and $b=0$, the Strichartz
estimates are due to Burq, Planchon, Stalker and Tahvildar-Zadeh \cite{BPST2003},
and the endpoint case was obtained by Bouclet and Mizutani
\cite{BoucletMizutani2018}. Fanelli, Felli, Fontelos and Primo \cite{FFFP2013}
obtained dispersive estimates in three dimensions. Killip, Miao, Visan, Zhang and
Zheng \cite{KMVZZ2018} compared the Sobolev spaces defined through $\Lc$ with the
classical ones, a tool that has become indispensable. Killip, Murphy,
Visan and Zheng \cite{KMVZ2017} obtained the focusing analogue of \cite{DHR2008}
for the cubic equation in three dimensions, under the conditions
\begin{equation}\label{eq:KMVZ}
\Mass(u_0)\Ea(u_0)<\Mass(Q_{a\wedge0})E_{a\wedge0}(Q_{a\wedge0}),\qquad
\|u_0\|_{L^{2}}\|u_0\|_{\dHa}<\|Q_{a\wedge0}\|_{L^{2}}\|Q_{a\wedge0}\|_{\dot H^1_{a\wedge0}},
\end{equation}
where $a\wedge0=\min\{a,0\}$. The appearance of $a\wedge0$ reflects a genuine
phenomenon. For $b=0$ and $a>0$, the sharp Gagliardo--Nirenberg constant coincides
with the one for $a=0$ and is not attained, since optimizing sequences escape to
spatial infinity, where the potential is negligible. In other words, for $b=0$ and
$a\ge0$, the threshold is given by the ground state of the equation without
potential. This phenomenon disappears when $b>0$. The decay of the weight restores
compactness, $\GN$ is attained for every $a>\aH$ (\cite[Lemma~3.2]{CamposGuzman2021},
Proposition \ref{prop:GN} and Remark \ref{rmk:attained}), and the threshold is
carried by a ground state of the equation with the potential.

Further developments include the intercritical range \cite{LuMiaoMurphy2018},
radial data in higher dimensions \cite{Zheng2018}, blow-up criteria
\cite{Dinh2021}, the critical coupling $a=\aH$
\cite{Suzuki2016,MNN2021}, and the uniqueness of ground states
\cite{MNN2021,YZZ2026}. In \cite{YZZ2026}, the solutions lying on the
mass-energy level of the ground state are also classified, for $N=3,4,5$ and
$\aH<a<0$. A different scale-critical perturbation, namely the
Aharonov--Bohm magnetic potential, was combined with an inhomogeneous
nonlinearity in \cite{MajdoubSaanouni2024}.

\emph{Both perturbations together.} The Cauchy problem for \eqref{eq:main} was
first studied by Suzuki \cite{Suzuki2016}, by the energy method of Okazawa, Suzuki
and Yokota \cite{Suzuki2012}. As recalled in
\cite{CamposGuzman2021}, problem \eqref{eq:main} is locally well posed in
$\Ha(\R^N)$ for $N\ge3$, $0<b<2$, $0<\alpha<\frac{4-2b}{N-2}$ and $a>\aH$, with a
solution $u\in C\big([0,T);H^1(\R^N)\big)\cap C^1\big([0,T);H^{-1}(\R^N)\big)$.
Every local solution extends globally in time in the defocusing case, and also in
the focusing case when $\alpha<\frac{4-2b}{N}$. The critical coupling $a=\aH$ is
treated in \cite{Suzuki2016}. The energy method does not show whether these
solutions belong to $L^q_{loc}\big(H_a^{1,r}\big)$ for every admissible pair
$(q,r)$, a property that is needed for scattering. Campos and Guzm\'an
\cite{CamposGuzman2021} proved it by Kato's method, which combines the contraction
mapping principle with Strichartz estimates, under additional restrictions on the
parameters.

Campos and Guzm\'an \cite{CamposGuzman2021} also developed the variational theory
of \eqref{eq:main}. They established the Gagliardo--Nirenberg inequality adapted
to $\Lc$ and to the weight,
\begin{equation}\label{eq:GNintro}
\int_{\R^{N}}\frac{|u|^{\alpha+2}}{|x|^{b}}\dd
\le\GN\,\|u\|_{\dHa}^{\frac{N\alpha+2b}{2}}\,\|u\|_{L^{2}}^{\frac{4-2b-(N-2)\alpha}{2}},
\qquad u\in\Ha(\R^{N}),
\end{equation}
and proved that its sharp constant is attained by a positive solution of the
stationary equation \eqref{eq:ellipticQ} below \cite[Lemma~3.2 and
Proposition~1.1]{CamposGuzman2021}. With this ground state, they obtained the
alternative between global existence and finite-time blow-up below the ground
state, the blow-up part being proved for data of finite variance or radial data
\cite[Theorem~1.3]{CamposGuzman2021}. They also proved small-data global results,
a scattering criterion in Strichartz norms and the existence of wave operators
\cite[Theorems~1.8--1.10 and~1.12]{CamposGuzman2021}. Building on these results, Hajaiej,
Luo and Wang \cite{HLW2024} considered the three-dimensional cubic model ($N=3$,
$\alpha=2$, $0<b<1$, $a>-\frac14$). In the focusing case, they proved scattering
below the ground state. In the defocusing case, they established global
well-posedness and scattering in $\Ha(\R^{3})$ for arbitrary data. Their proof
adapts the virial--Morawetz method of \cite{DodsonMurphy2017,CamposCardoso2022}.
In the attractive case $a<0$, additional care is needed, because only a weighted
dispersive estimate is available \cite{FFFP2013}.

In a parallel direction, Boulaaras, Ghanmi and the second author \cite{BGS2023}
proved energy scattering below the ground state for \eqref{eq:main} in three space
dimensions, for $a>0$, for every $0<b<2$ and for every power in the intercritical
range \eqref{eq:range}, without any radial assumption. Their argument follows
Dodson and Murphy \cite{DodsonMurphy2017}. A scattering criterion based on the
dispersive estimate \eqref{eq:disp}, which is where the sign condition on $a$
enters, is combined with a Morawetz estimate that exploits the decay of the weight
$|x|^{-b}$ at infinity. The threshold in \cite{BGS2023} is expressed through a
solution $\varphi$ of \eqref{eq:ellipticQ} which is an optimizer of
\eqref{eq:GNintro} and satisfies the Pohozaev identities. It is formulated either
by a condition on the potential energy, which is not conserved,
\begin{equation*}
\sup_{t}\,\Mass(u(t))^{\sigma}\int_{\R^{3}}\frac{|u(t)|^{\alpha+2}}{|x|^{b}}\dd
<\Mass(\varphi)^{\sigma}\int_{\R^{3}}\frac{|\varphi|^{\alpha+2}}{|x|^{b}}\dd ,
\end{equation*}
or by the conserved conditions
$\Mass(u_0)^{\sigma}\Ea(u_0)<\Mass(\varphi)^{\sigma}\Ea(\varphi)$ and
$\|u_0\|_{L^2}^{\sigma}\|u_0\|_{\dHa}<\|\varphi\|_{L^2}^{\sigma}\|\varphi\|_{\dHa}$.
The exponent is the scale-invariant one, $\sigma=\frac{1-s_c}{s_c}$ of
\eqref{eq:sigma} below. As explained in \cite{BGS2023}, this complements the work
of An, Jang and Kim \cite{AJK2023}, where the alternative between global existence
and blow-up below the ground state is established for the same model. An, Mun and Kim \cite{AMK2024} proved the
stability of the set of ground states in the mass-subcritical case, and the
instability by blow-up of the ground-state standing waves in the mass-critical and
intercritical cases. Blow-up of the critical Lebesgue norm, including for some
non-radial data, was studied by Campos, Cardoso and Farah \cite{CCF2023}, and
non-global solutions with non-radial data of negative energy were obtained by Bai
and the second author \cite{BaiSaanouni2025}. In the mass-critical case,
Boulaaras and the second author \cite{BS2025} proved finite-time blow-up for data
of negative energy, without any radial or finite-variance assumption. The effect of a time-dependent
damping term on global existence and scattering was studied by Hamouda and the
present authors \cite{HMS2025}.

The local well-posedness theory of \eqref{eq:main} in the energy space
$\Ha(\R^N)$, $N\ge 3$, is summarized in Table~\ref{tab_LWP}. Recall that
$\Ha(\R^N)=H^1(\R^N)$, with equivalent norms, for $a>\aH$.

\begin{table}[!htbp]
\centering
\small
\caption{Local well-posedness of \eqref{eq:main} in the energy space.}
\label{tab_LWP}
\renewcommand{\arraystretch}{2.1}
\setlength{\tabcolsep}{2.5pt}
\setlength{\aboverulesep}{0pt}
\setlength{\belowrulesep}{0pt}
\begin{tabular}{l l l c l}
\toprule[1pt]
\rowcolor{tabhead}
\textbf{Coupling $a$} & \textbf{Weight $b$} & \textbf{Power $\alpha$} & \textbf{Ref.} & \textbf{Comment} \\
\midrule[0.6pt]
$a\ge {\aH}$
  & $0<b<2$
  & $0<\alpha<\frac{4-2b}{N-2}$
  & \cite{Suzuki2012,Suzuki2016}
  & energy method \\
\rowcolor{tabrow}
$a> {\aH}$
  & $0\le b<1$
  & $\frac{2-2b}{N}<\alpha\,{\leq}\frac{2-2b}{N-2}$
  & \cite{CamposGuzman2021}
  & Kato's method \\
$a> {\aH}+\Big(\frac{\alpha(N-2)-2+2b}{2(\alpha+1)}\Big)^{2}$
  & $0\le b<\min\{\frac N2,2\}$
  & $\frac{2(1-b)_{+}}{N-2}<\alpha<\frac{4-2b}{N-2}$
  & \cite{CamposGuzman2021}
  & Kato's method \\
\rowcolor{tabrow}
$a>\frac{(N+2-2b)^{2}-4}{(N+2-2b)^{2}}\, {\aH}$
  & $0<b<\frac4N$
  & $\alpha=\frac{4-2b}{N-2}$
  & \cite{Kim}
  & energy-critical case \\
\bottomrule[1pt]
\end{tabular}

\end{table}

\emph{Normalized solutions.} The second half of \cite{HLW2024} takes a different
perspective. Instead of prescribing the frequency of a standing wave, one fixes
its mass and seeks critical points of $\Ea$ on the sphere
\begin{equation}\label{eq:Sc}
S(c):=\Big\{u\in\Ha(\R^{N}):\|u\|_{L^{2}}=c\Big\}.
\end{equation}
This is the natural formulation when the underlying physical model fixes the
number of particles rather than the chemical potential. The study of such
normalized solutions was initiated by Jeanjean \cite{Jeanjean1997} in the
mass-supercritical regime, and it has since become a substantial area of research;
see, for instance, Soave \cite{Soave2020a,Soave2020b} and the references therein.
Nonlocal operators with inhomogeneous nonlinearities were considered by Bhimani,
Hajaiej, Haque and Luo \cite{BHHL2023}. Normalized solutions for equations with an
inverse-square (Hardy) potential were obtained, mainly at the Sobolev-critical
exponent, by Fan, Li and Tang \cite{FLT2024}; see also the references therein. Normalized standing
waves for an inhomogeneous nonlinearity with a spatially growing weight were
studied by the present authors in \cite{MS2026}.

In the range \eqref{eq:range}, the functional $\Ea$ is unbounded from below on
$S(c)$. One therefore minimizes on the Pohozaev-type set
\begin{equation}\label{eq:Vc}
V(c):=\Big\{u\in S(c):\Qa(u)=0\Big\},\qquad
\Qa(u):=\frac{d}{dt}\Big|_{t=1}\Ea(u^{t}),\quad u^{t}(x):=t^{N/2}u(tx),
\end{equation}
which contains every sufficiently regular solution of the stationary equation. A
direct computation gives
\begin{equation}\label{eq:Qdef}
\Qa(u)=\|u\|_{\dHa}^{2}-\frac{N\alpha+2b}{2(\alpha+2)}
\int_{\R^{N}}\frac{|u|^{\alpha+2}}{|x|^{b}}\dd .
\end{equation}
Hajaiej, Luo and Wang studied the level
\begin{equation}\label{eq:mc}
m_a(c):=\inf_{u\in V(c)}\Ea(u).
\end{equation}
In \cite{HLW2024}, this level is denoted by $M(c)$. They proved that it is positive
and attained for every $c>0$, in every dimension $N\ge3$, for $0<b<2$ and for every
$\alpha$ in \eqref{eq:range} \cite[Theorem~1.6]{HLW2024}. This yields a normalized
solution $u_c$ with some frequency $\omega_c>0$. They also showed that no such
solution exists in the defocusing case \cite[Theorem~1.8]{HLW2024}. Along the way,
they computed the maximum of the energy along each fibre in terms of $A(u)$ and
$B(u)$, characterized $m_a(c)$ as a minimax level, and proved that $c\mapsto
m_a(c)$ is strictly decreasing \cite[Lemmas~5.2--5.4]{HLW2024}. Their argument, however, leaves open whether $u_c$ is a
\emph{ground state}, that is, whether it has least action among all nontrivial
solutions with the same frequency \cite[Remark~1.7]{HLW2024}.

\subsection{The question}\label{ssec:question}

The two parts of \cite{HLW2024} thus produce two objects which, a priori, have
nothing to do with each other. On the dynamical side, the threshold is expressed
through a ground state $Q$ of
\begin{equation}\label{eq:ellipticQ}
\Lc Q+Q=|x|^{-b}|Q|^{\alpha}Q ,
\end{equation}
by means of products of mass, energy and kinetic energy as in \eqref{eq:HR} and
\eqref{eq:KMVZ}. On the variational side, the natural object is the number
$m_a(c)$, which depends only on the prescribed mass. The purpose of this paper is
to connect the two, and more precisely to answer the following question.

\begin{question}\label{q:main}
Is the normalized minimizer constructed in \cite{HLW2024} a ground state of the
stationary problem, and is the level $m_a(c)$ a sharp threshold separating
scattering from blow-up among data of mass $c$?
\end{question}

Both answers are affirmative, with the restrictions on the scattering part stated
in Subsection \ref{ssec:results}. The key to both answers is that $m_a(c)$ can
be computed in closed form. Along the $L^2$-preserving dilations $u\mapsto u^{t}$, the energy
has a single maximum. Its value depends on $u$ only through the mass and the
Gagliardo--Nirenberg quotient, since the kinetic energy cancels out. The
constrained minimax problem thus reduces to the optimization of that quotient.
Once this is known, the variational and dynamical pictures merge. Indeed,
the level $m_a(c)$ is the Holmer--Roudenko threshold written in the variables
adapted to a problem with fixed mass, while the functional $\Qa$ that defines
$V(c)$ is one eighth of the second time derivative of the variance. This formulation is convenient for
three reasons. It is scale invariant by construction, so no bookkeeping of
critical exponents is needed. It does not depend on any normalization of the
ground state, nor on its uniqueness. Finally, it makes the blow-up half of the
dichotomy as transparent as the scattering half.

\subsection{Notation, definitions and standing assumption}\label{ssec:notation}

We use the abbreviations
\begin{equation}\label{eq:betasig}
\beta:=\frac{N\alpha+2b}{2},\qquad
\sigma_c:=\frac{4-2b-(N-2)\alpha}{2},\qquad
\gamma:=\frac{2\sigma_c}{\beta-2}=\frac{2\big(4-2b-(N-2)\alpha\big)}{N\alpha+2b-4}.
\end{equation}
Thus \eqref{eq:range} is equivalent to the two conditions $\beta>2$ and
$\sigma_c>0$. We set
\begin{equation}\label{eq:K}
K_{\alpha,N,b}:=\frac{\beta-2}{2\beta}\Big(\frac{\alpha+2}{\beta}\Big)^{\frac{2}{\beta-2}}
=\frac{N\alpha+2b-4}{2(N\alpha+2b)}\Big(\frac{2\alpha+4}{N\alpha+2b}\Big)^{\frac{4}{N\alpha+2b-4}} .
\end{equation}
We shall also use the identity, established in Lemma \ref{lem:sigma},
\begin{equation}\label{eq:sigma}
\sigma:=\frac{\gamma}{2}=\frac{1-s_c}{s_c}.
\end{equation}
This identity shows that $\gamma$ is determined by the scaling
\eqref{eq:scaling}, and not by the particular choice of constraint. We denote by
$\GN=\GN(N,\alpha,b)$ the best constant in \eqref{eq:GNintro}, and we write
$\|\cdot\|_{r}$ for the norm of $L^{r}(\R^{N})$.

For $u\in\Ha(\R^N)$, we introduce the functionals
\begin{equation}\label{eq:ABm}
A(u):=\|u\|_{\dHa}^{2},\qquad
B(u):=\int_{\R^N}\frac{|u|^{\alpha+2}}{|x|^{b}}\dd,\qquad
\Mass(u):=\|u\|_{2}^{2}.
\end{equation}
Then, with $\beta$ defined in \eqref{eq:betasig}, we obtain the following identities
\begin{equation}\label{eq:EQAB}
\Ea(u)=\frac{A(u)}{2}-\frac{B(u)}{\alpha+2},\qquad
\Qa(u)=A(u)-\frac{\beta}{\alpha+2}B(u),\qquad
\Sw(u)=\Ea(u)+\frac{\omega}{2}\Mass(u).
\end{equation}
We also write $\langle u,h\rangle_a:=\int_{\R^N}\big(\nabla u\cdot\nabla\bar h+a|x|^{-2}u\bar h\big)\dd$
for the sesquilinear form associated with $A$. For $t>0$, $u^{t}(x):=t^{N/2}u(tx)$
denotes the $L^{2}$-preserving dilation, and we shall constantly use the scaling
relations
\begin{equation}\label{eq:scalingAB}
\Mass(u^{t})=\Mass(u),\qquad A(u^{t})=t^{2}A(u),\qquad B(u^{t})=t^{\beta}B(u).
\end{equation}
The last one follows from the change of variables $y=tx$ and the homogeneity of
$|x|^{-b}$, since $\frac{N(\alpha+2)}{2}+b-N=\frac{N\alpha+2b}{2}=\beta$.

\begin{definition}\label{def:weak}
Let $\omega\in\R$. A function $v\in\Ha(\R^N)$ is a (weak) solution of
\eqref{eq:stationary} if
\begin{equation*}
\langle v,h\rangle_a+\omega\int v\bar h\dd=\int|x|^{-b}|v|^{\alpha}v\,\bar h\dd
\qquad\text{for every }h\in H^1(\R^N).
\end{equation*}
All integrals are finite by \eqref{eq:hardy} and Lemma \ref{lem:weighted} below.
\end{definition}

In terms of these functionals, \eqref{eq:GNintro} reads
$B(u)\le\GN\,A(u)^{\beta/2}\|u\|_{2}^{\sigma_c}$, and
\begin{equation*}
\GN=\sup_{u\in\Ha(\R^N)\setminus\{0\}}J(u),\qquad
J(u):=\frac{B(u)}{A(u)^{\beta/2}\|u\|_{2}^{\sigma_c}}
\end{equation*}
(see Proposition~\ref{prop:GN}). A function $W\in\Ha\setminus\{0\}$ is called an
\emph{optimizer} if $J(W)=\GN$, that is, if equality holds in
\eqref{eq:GNintro} with $u=W$. Throughout, solutions of the stationary equation
are understood in the weak sense of Definition \ref{def:weak}.

\begin{definition}\label{def:gs}
Let $\omega\in\R$. A weak solution $u\ne0$ of \eqref{eq:stationary} is a
\emph{ground state} at the frequency $\omega$ if $\Sw(u)\le\Sw(v)$ for every weak
solution $v\ne0$ of \eqref{eq:stationary} with the same $\omega$. A
\emph{normalized ground state} of mass $c$ is a minimizer of the problem
\eqref{eq:mc}.
\end{definition}

We denote by $\Sigma:=\{u\in\Ha:|x|u\in L^{2}\}$ the space of data with
finite variance. To state our dynamical results without hidden hypotheses, we isolate the
properties of the Cauchy problem that we use.

\begin{assumption}\label{ass:cauchy}
For every $u_0\in\Ha(\R^{N})$, there are $T_{\pm}=T_{\pm}(u_0)\in(0,\infty]$ and a
unique maximal solution
\begin{equation*}
u\in C\big((-T_-,T_+);\Ha(\R^{N})\big)\cap C^1\big((-T_-,T_+);H^{-1}(\R^{N})\big)
\end{equation*}
of \eqref{eq:main}. This solution conserves the mass and the energy
\eqref{eq:conserved} and depends continuously on the data. Moreover, the blow-up
alternative holds: if $T_{+}<\infty$, then
$\lim_{t\uparrow T_{+}}\|u(t)\|_{\dHa}=\infty$, and similarly backwards in time.
If in addition $u_0\in\Sigma$, then $u(t)\in\Sigma$ for all $t\in(-T_-,T_+)$, the
map $t\mapsto\||x|u(t)\|_{2}$ is continuous, and the virial identity
\eqref{eq:virial1} below holds.
\end{assumption}

We call the statements on $u_0\in\Ha$ the first part of Assumption
\ref{ass:cauchy}, and the statements on $u_0\in\Sigma$ its second part. Lemma
\ref{lem:invariance}, and hence the global existence statement in Theorem
\ref{thm:dichotomy}(i), uses only the first part.
By Proposition~\ref{prop:lwp} below, Assumption~\ref{ass:cauchy} holds in the
whole range \eqref{eq:range}. We nevertheless keep it as a hypothesis in the
dynamical statements, in order to isolate exactly which properties of the Cauchy
problem are used.

\subsection{Main results}\label{ssec:results}

\begin{theorem}\label{thm:level}
Let $N\ge3$, $a>\aH$, and $0<b<2$, and assume that \eqref{eq:range} holds.
\begin{enumerate}
\item[(i)] For every $c>0$,
\begin{equation}\label{eq:levelformula}
m_a(c)=K_{\alpha,N,b}\;\GN^{-\frac{4}{N\alpha+2b-4}}\;c^{-\gamma}.
\end{equation}
In particular $c\mapsto c^{\gamma}m_a(c)$ is constant, and $m_a$ is real
analytic and strictly decreasing on $(0,\infty)$, with $m_a(c)\to\infty$ as
$c\downarrow0$ and $m_a(c)\to0$ as $c\to\infty$.
\item[(ii)] The constant $\GN$ is attained, and $m_a(c)$ is attained for every
$c>0$. A function $u\in V(c)$ satisfies $\Ea(u)=m_a(c)$ if and only if it is an
optimizer of \eqref{eq:GNintro}. Consequently the set of minimizers of $m_a(c)$
is $\{W^{t_W}:W\ \text{an optimizer with}\ \|W\|_{L^{2}}=c\}$, where $t_W>0$ is
the unique dilation parameter for which $W^{t_W}\in V(c)$.
\item[(iii)] Every minimizer $u_c$ of $m_a(c)$ is a ground state of
\begin{equation}\label{eq:stationary}
\Lc u+\omega u=|x|^{-b}|u|^{\alpha}u
\end{equation}
at the frequency
\begin{equation}\label{eq:omegac}
\omega_c=\Lambda_{\alpha,N,b}\,\GN^{-\frac{4}{N\alpha+2b-4}}\;
c^{-\frac{4\alpha}{N\alpha+2b-4}},\qquad
\Lambda_{\alpha,N,b}:=\frac{\sigma_c}{\alpha+2}\Big(\frac{\alpha+2}{\beta}\Big)^{\frac{\beta}{\beta-2}},
\end{equation}
that is, $S_{\omega_c}(u_c)\le S_{\omega_c}(v)$ for every nontrivial solution $v\in\Ha$ of
\eqref{eq:stationary}, where $\Sw:=\Ea+\frac{\omega}{2}\Mass$.
Conversely, every ground state $v$ of \eqref{eq:stationary}, at any frequency
$\omega>0$, is a minimizer of $m_a(\|v\|_{L^{2}})$, and $\omega=\omega_{\|v\|_{L^2}}$.
The map $c\mapsto\omega_c$ is a decreasing bijection of $(0,\infty)$ onto itself.
\end{enumerate}
\end{theorem}

\begin{remark}\rm\label{rmk:limiting}
The case $a=0$ is covered by Theorem \ref{thm:level}. The case $b=0$ is excluded,
and the proof shows exactly where $b>0$ is needed.
\begin{enumerate}[label=(\roman*)]
\item\emph{The case $a=0$ and $0<b<2$.} Then $\Lc=-\Delta$, and Theorem
\ref{thm:level} applies without change. The constant $\mathcal{C}_0$ is the sharp
constant of Farah \cite{Farah2016}, and \eqref{eq:levelformula} gives the
normalized level of the inhomogeneous equation without potential in terms of it.
For $N\ge3$, the positive radial solution of \eqref{eq:ellipticQ} with $a=0$ is
unique \cite{Yanagida1991,GenoudStuart2008}. By Theorem \ref{thm:equiv}, the
conditions $\Ea(u_0)<m_a(c)$ and $\Qa(u_0)\gtrless0$, with $a=0$, are equivalent to
the scale-invariant conditions used for this equation in
\cite{Farah2016,FarahGuzman2017,MiaoMurphyZheng2021,CamposCardoso2022}.
Hence the global existence, blow-up and scattering results proved there hold
below $m_0(c)$, in the corresponding ranges of $(N,\alpha,b)$.

\item\emph{The case $b=0$.} In Theorem \ref{thm:level}, the hypothesis $b>0$ is
used only through the compactness property of Lemma \ref{lem:weighted}(b), that
is, only for the attainment of $\GN$. Part (i) holds for $b=0$ and every
$a>\aH$, since its proof uses only the finiteness of $\GN$ and the scaling
relations \eqref{eq:scalingAB}. Parts (ii) and (iii) hold for $b=0$ whenever
$\GN$ is attained. Three situations occur.
First, if $a=b=0$, then $\mathcal{C}_0$ is attained by the ground state of Weinstein
\cite{Weinstein1983} and its translates. Theorem \ref{thm:level} then reduces to
classical facts on normalized solutions \cite{Jeanjean1997}, the dichotomy below
the level is that of
\cite{HolmerRoudenko2008,DHR2008,FXC2011}, and the
strong instability is due to Berestycki and Cazenave \cite{BerestyckiCazenave1981}.
Second, if $b=0$ and $\aH<a<0$, then $\GN>\mathcal{C}_0$, and $\GN$ is attained
\cite{KMVZ2017,MNN2021}. Parts (ii) and (iii) remain valid, and the ground state is
unique up to phase by \cite{MNN2021}.
Third, if $b=0$ and $a>0$, the picture changes. Since $\|u\|_{\dHa}$ increases with $a$,
we have $\GN\le\mathcal{C}_0$. Translating Weinstein's ground state $Q_0$ to
infinity, we get $J\big(Q_0(\cdot-y)\big)\to\mathcal{C}_0$ as $|y|\to\infty$,
because $\int|Q_0(x-y)|^2|x|^{-2}\dd\to0$. Hence $\GN=\mathcal{C}_0$. This constant
is not attained. Indeed, an optimizer $W$ would satisfy
$\mathcal{C}_0=J(W)<J_0(W)\le\mathcal{C}_0$, where $J_0$ denotes the quotient
\eqref{eq:GNquotient} with $a=0$. Consequently, $m_a(c)=m_0(c)$ for every $a\ge0$,
and, by the argument used in the proof of part (ii), $m_a(c)$ is not attained for
$a>0$. The proof of Theorem \ref{thm:equiv} still applies, with $Q$ replaced by
$Q_0$, and it gives exactly the conditions \eqref{eq:KMVZ} with $a\wedge0=0$. This
explains the appearance of $a\wedge0$ in \cite{KMVZ2017}. In this case there is no
minimizer, and the sharpness statement of Theorem \ref{thm:dichotomy} has no
counterpart in the form given there.
\end{enumerate}
\end{remark}
\begin{remark}\rm \label{rmk:sharphyp}
The assumptions $N\ge3$ and \eqref{eq:range} are essential for the result in its
present form. We briefly discuss the corresponding borderline and complementary
cases. The case $b=0$ is considered in Remark \ref{rmk:limiting}, while the
critical coupling $a=\aH$ is treated in Subsection \ref{ssec:hardy}.
\begin{enumerate}[label=(\roman*)]
\item \emph{The mass-critical case $\alpha=\frac{4-2b}{N}$.} In this case
$\beta=2$, and therefore the exponent of $\GN$ in \eqref{eq:levelformula} is not defined. By
\eqref{eq:scalingAB}, the fibering map reduces to
$h_u(t)=t^{2}\Ea(u)$, and hence it has no interior maximum. Proposition
\ref{prop:masscritical} shows that $V(c)=\emptyset$ for $c<c^*$, while
$m_a(c)=0$ for $c\ge c^*$, where $c^*$ is given by \eqref{eq:cstar}. Thus,
parts (i)--(iii) no longer hold in the stated form, and the relevant threshold
is determined by the mass alone. Moreover,
$c^*=\|Q\|_{2}$, where $Q$ is the ground state from
\cite[Lemma~3.2]{CamposGuzman2021}. This agrees with
\cite[Theorem~1.2]{CamposGuzman2021}.

\item \emph{The mass-subcritical case $0<\alpha<\frac{4-2b}{N}$.} Here
$\beta<2$. The unique critical point $t_u$ of $h_u$ is a strict global
minimum, and $\Ea<0$ on $V(c)$. Consequently,
$m_a(c)=\inf_{S(c)}\Ea$. The computation of $h_u(t_u)$ in Lemma
\ref{lem:fibering} and the proofs of parts (i) and (ii) therefore remain
unchanged, the maximum over $t$ being replaced by a minimum. In particular,
\eqref{eq:levelformula} still holds, now with $K_{\alpha,N,b}<0$ and
$\gamma<0$. Hence $m_a(c)<0$, $m_a(c)\to0$ as $c\downarrow0$, and
$m_a(c)\to-\infty$ as $c\to\infty$. The minimizers are again precisely the
optimizers of \eqref{eq:GN} with mass $c$ that belong to $V(c)$. In this
regime, $m_a(c)$ is the minimum of the energy on the sphere $S(c)$. This is
the natural setting for orbital stability, rather than that of a dynamical
threshold. Moreover, every solution is global; see
\cite{Suzuki2016,CamposGuzman2021}.

\item \emph{The energy-critical case
$\alpha=\frac{4-2b}{N-2}$.} In this case $\sigma_c=0$ and $\gamma=0$, and
\eqref{eq:GN} becomes the Hardy--Sobolev type inequality
$B(u)\le\GN A(u)^{\frac{\alpha+2}{2}}$. This inequality is invariant under the
scaling $u\mapsto\lambda^{\frac{N-2}{2}}u(\lambda\cdot)$, which preserves
$A$ and $B$ but not the mass. Consequently, the level given by
\eqref{eq:maxvalue} is independent of $c$. At the same time, the compactness
property in Lemma \ref{lem:weighted}(b) is lost, since optimizing sequences may
either concentrate or spread out.

The argument used in part (ii) shows that $m_a(c)$ is attained for one, and
equivalently for every, $c>0$, if and only if \eqref{eq:GN} admits an optimizer
in $\dHa$ that belongs to $L^{2}(\R^N)$. The natural threshold is then given by
the energy of such an optimizer, independently of the mass, as in
Kenig and Merle \cite{KenigMerle2006} and, for $a=0$, Guzm\'an and Murphy
\cite{GuzmanMurphy2021}. The Cauchy theory is given in \cite{Kim}; see also
Table~\ref{tab_LWP}. For normalized solutions at the Sobolev-critical exponent,
we refer to \cite{Soave2020b}. If
$\alpha>\frac{4-2b}{N-2}$, the functional $B$ is no longer controlled on
$\Ha$, and consequently $\GN=+\infty$.

\item \emph{The dimension.} The assumption $N\ge3$ enters through the Hardy
inequality \eqref{eq:hardy}, which yields \eqref{eq:equivnorm} and the
identification $\Ha=H^1$. It also enters through the Sobolev embedding used in
Lemma \ref{lem:weighted} and through the Cauchy theory. For $N=2$, we have
$\aH=0$, and \eqref{eq:hardy} fails. When $a>0$, the form domain of $\Lc$ is
then a proper subspace of $H^1(\R^2)$, and \eqref{eq:range} becomes
$\alpha>2-b$. For $N=1$ and $a\ne0$, the form domain of $\Lc$ is
$H^1_0(\R\setminus\{0\})$, so the problem decouples into two problems on the
half-lines.

The variational argument in Section \ref{sec:variational} uses only
\eqref{eq:GN}, \eqref{eq:scalingAB}, and the compactness statement in
Lemma \ref{lem:weighted}(b). We therefore expect that it extends to $N=2$ and
$a>0$, with $\Ha$ replaced by the form domain of $\Lc$. The corresponding
dynamical results would, however, require an appropriate Cauchy theory in this
setting.
\end{enumerate}
\end{remark}

Part (iii) of Theorem~\ref{thm:level} answers the question raised in
\cite[Remark~1.7]{HLW2024}. Part (ii) recovers the existence of normalized
minimizers obtained in \cite[Theorem~1.6]{HLW2024}, with a shorter proof, and it
identifies all of them.

The argument is relatively short. In the absence of translation invariance,
minimizing sequences may escape to spatial infinity, as occurs when $b=0$ and
$a>0$. For $b>0$, the decay of the weight $|x|^{-b}$ rules out this loss of
compactness, and consequently the associated Gagliardo--Nirenberg quotient
admits a maximizer; see \cite[Lemma~3.2]{CamposGuzman2021}. The explicit formula
for $m_a(c)$ then transfers this property to the normalized minimization
problem.

Moreover, \eqref{eq:levelformula} shows that $c\mapsto m_a(c)$ is a pure power
function, which provides a direct explanation of the strict monotonicity
proved in \cite{HLW2024} by a scaling argument.

The next results concern the dynamics. For the scattering statement, we restrict
ourselves to the parameters covered by the scattering theory of \cite{HLW2024},
namely
\begin{equation}\label{eq:dynrange}
N=3,\qquad\alpha=2,\qquad0<b<1,\qquad a>-\tfrac14 ,
\end{equation}
for which $s_c=\frac{1+b}{2}\in(\frac12,1)$ and $\sigma=\frac{1-b}{1+b}$. Recall
that $\Sigma$ denotes the space of data with finite variance.

\begin{theorem}\label{thm:dichotomy}
Let $N\ge3$, $a>\aH$ and $0<b<2$. Assume that \eqref{eq:range} and Assumption~\ref{ass:cauchy} hold. Let $u_0\in\Ha(\R^{N})\setminus\{0\}$ satisfy
\begin{equation}\label{eq:below}
\Ea(u_0)<m_a(c),\qquad c:=\|u_0\|_{L^{2}} .
\end{equation}
Then $\Qa(u_0)\ne0$, the sign of $\Qa(u(t))$ does not depend on $t$, and the
following alternative holds.
\begin{enumerate}
\item[(i)] If $\Qa(u_0)>0$, the solution of \eqref{eq:main} is global. If
moreover \eqref{eq:dynrange} holds, it scatters in $\Ha(\R^{3})$: there exist
$u_\pm\in\Ha(\R^{3})$ such that $\|u(t)-e^{-it\Lc}u_\pm\|_{\Ha}\to0$ as
$t\to\pm\infty$.
\item[(ii)] If $\Qa(u_0)<0$ and $u_0\in\Sigma$, the solution blows up in finite
time in both time directions, and
\begin{equation}\label{eq:blowup}
\frac{d^{2}}{dt^{2}}\int_{{\R^{N}}}|x|^{2}|u(t,x)|^{2}\dd
\le 8\beta\big(\Ea(u_0)-m_a(c)\big)<0
\end{equation}
on the whole lifespan.
\end{enumerate}
Moreover, the condition \eqref{eq:below} is sharp: for every $c>0$ there is
$u_0\in V(c)\cap\Sigma$ with $\Ea(u_0)=m_a(c)$ whose solution is the global,
non-scattering standing wave $e^{i\omega_ct}u_c$.
\end{theorem}

Global existence in part (i), part (ii) and the sharpness statement are proved
here under Assumption \ref{ass:cauchy}. This assumption holds in the whole range
\eqref{eq:range} (Proposition \ref{prop:lwp}). The scattering statement in part
(i) relies on the scattering criterion of \cite{HLW2024}, used in the
scale-invariant form explained in Remark \ref{rmk:exponent} (see Proposition
\ref{prop:scatcrit}). It is the only statement of the paper that rests on an
external result used in a form that differs from the published one.
 The
dichotomy is exhaustive, because $\Qa$ cannot vanish below the level: $\Qa(u_0)=0$
would put $u_0$ in $V(c)$ and force $\Ea(u_0)\ge m_a(c)$. 
\begin{remark}\rm \label{rmk:negenergy}
Since $m_a(c)>0$, Theorem \ref{thm:dichotomy} applies to every $u_0\ne0$ with
$\Ea(u_0)\le0$. For such data, \eqref{eq:conserved} and \eqref{eq:Qdef} give
$\Qa(u_0)=\beta\Ea(u_0)-\frac{\beta-2}{2}\|u_0\|_{\dHa}^{2}<0$. Hence every
$u_0\in\Sigma\setminus\{0\}$ with $\Ea(u_0)\le0$ blows up in finite time in both
time directions. This is the criterion of Glassey \cite{Glassey1977}; see
\cite{Farah2016} for $a=0$ and \cite{Dinh2021} for $b=0$. Conversely, part (i)
concerns only data of positive energy, since $\Qa(u_0)>0$ forces
$0<\Ea(u_0)<m_a(c)$ (Lemma \ref{lem:invariance}). The substance of Theorem
\ref{thm:dichotomy} thus lies in the range $0<\Ea(u_0)<m_a(c)$, where the energy
alone does not determine the sign of $\Qa(u_0)$.
\end{remark}

The relation with the
Holmer--Roudenko formulation is as follows.

\begin{theorem}\label{thm:equiv}
Let $N\ge3$, $a>\aH$, and $0<b<2$. Assume that \eqref{eq:range} holds, and let $Q=Q_{a,b}$ be a ground state
solution of \eqref{eq:ellipticQ} (such a ground state exists by Theorem
\ref{thm:level}(iii)).
Then for every
$u_0\in\Ha\setminus\{0\}$ of mass $c=\|u_0\|_{L^{2}}$,
\begin{equation}\label{eq:equivE}
\Ea(u_0)<m_a(c)
\qquad\Longleftrightarrow\qquad
\Mass(u_0)^{\sigma}\Ea(u_0)<\Mass(Q)^{\sigma}\Ea(Q),
\end{equation}
and, when these conditions hold,
\begin{equation}\label{eq:equivQ}
\Qa(u_0)>0
\qquad\Longleftrightarrow\qquad
\|u_0\|_{L^{2}}^{\sigma}\|u_0\|_{\dHa}<\|Q\|_{L^{2}}^{\sigma}\|Q\|_{\dHa},
\end{equation}
with $\sigma=\frac{1-s_c}{s_c}$ as in \eqref{eq:sigma}. If moreover
$\Ea(u_0)>0$, the right-hand side of \eqref{eq:equivE} can be written in the
familiar form
$\Mass(u_0)^{1-s_c}\Ea(u_0)^{s_c}<\Mass(Q)^{1-s_c}\Ea(Q)^{s_c}$.
\end{theorem}
\begin{remark}\rm 
In the potential-well theory of Payne and Sattinger \cite{ps}, the region below
the ground-state level is divided by the sign of a functional into a stable set
and an unstable set, and the two conditions are used together. Theorem
\ref{thm:equiv} gives this combined form at once:
\begin{multline*}
\Big[\Ea(u_0)<m_a(c)\ \text{and}\ \Qa(u_0)>0\Big]\\
\iff
\Big[\Mass(u_0)^{\sigma}\Ea(u_0)<\Mass(Q)^{\sigma}\Ea(Q)\ \text{and}\
\|u_0\|_{2}^{\sigma}\|u_0\|_{\dHa}<\|Q\|_{2}^{\sigma}\|Q\|_{\dHa}\Big].
\end{multline*}
The same equivalence holds with $\Qa(u_0)<0$ on the left and
$\|u_0\|_{2}^{\sigma}\|u_0\|_{\dHa}>\|Q\|_{2}^{\sigma}\|Q\|_{\dHa}$ on the right.
This last inequality is strict. Indeed, when $\Ea(u_0)<m_a(c)$ and
$\Qa(u_0)<0$, the proof of Lemma \ref{lem:coercivity}(i) gives
$\frac{\beta-2}{2\beta}\|u_0\|_{\dHa}^{2}>m_a(c)$, that is,
$\|u_0\|_{\dHa}^{2}>A_*(c)$, with $A_*(c)$ given by \eqref{eq:Astar}. The separate
formulation of Theorem \ref{thm:equiv} is slightly stronger, since
\eqref{eq:equivE} requires no sign condition. The energy condition must be
written with $m_a(c)$, or equivalently with the factor $\Mass^{\sigma}$. The
condition $\Ea(u_0)<\Ea(Q)$ alone is not invariant under \eqref{eq:scaling}.
\end{remark}
\begin{remark}\rm \label{rmk:exponent}
Both sides of \eqref{eq:equivE} and of \eqref{eq:equivQ} are invariant under
\eqref{eq:scaling}, as they should be. The exponent $\sigma$ equals $1$ only when
$s_c=\frac12$, which in the three-dimensional cubic case means $b=0$. The
threshold conditions in \cite[Theorem~1.1]{HLW2024} are written with the exponent
$1$, as in \eqref{eq:HR}, also when $b>0$. In that literal form, such conditions
cannot be sufficient for scattering. Indeed, if $Q$ solves \eqref{eq:ellipticQ},
then $Q_\omega(x):=\omega^{\frac{2-b}{2\alpha}}Q(\sqrt\omega\,x)$ solves
\eqref{eq:stationary}, and one checks that
$\Mass(Q_\omega)\Ea(Q_\omega)=\omega^{1-2s_c}\Mass(Q)\Ea(Q)$ and
$\|Q_\omega\|_{L^2}\|Q_\omega\|_{\dHa}=\omega^{\frac{1-2s_c}{2}}\|Q\|_{L^2}\|Q\|_{\dHa}$.
When $s_c>\frac12$ and $\omega>1$, both products are strictly below those of
$Q$, yet $e^{i\omega t}Q_\omega$ does not scatter. The conditions must therefore
be read in the scale-invariant form $\Mass^{1-s_c}\Ea^{s_c}$, as in
\cite{FarahGuzman2017,MiaoMurphyZheng2021,CamposGuzman2021}. This is how we use
\cite[Theorem~1.1]{HLW2024}; see Proposition \ref{prop:scatcrit} and the
discussion following it. Theorem \ref{thm:dichotomy} is free of this ambiguity,
because $m_a(c)$ has the correct homogeneity by construction.
\end{remark}

Combining Theorem \ref{thm:dichotomy}(ii) with the fibering map, we obtain the
instability of the standing waves produced by the variational construction. Recall that a standing wave $e^{i\omega t}\varphi$ is
\emph{strongly unstable} if for every $\eps>0$ there is $v_0$ with
$\|v_0-\varphi\|_{\Ha}<\eps$ whose solution blows up in finite time.

\begin{theorem}\label{thm:instability}
Let $N\ge3$, $a>\aH$ and $0<b<2$, assume \eqref{eq:range} and Assumption
\ref{ass:cauchy}, and let $u_c$ be a minimizer of $m_a(c)$. Then
$u_c\in\Sigma$, and the standing wave $e^{i\omega_ct}u_c$ is strongly unstable.
Moreover the destabilizing data can be taken on the same mass sphere: for every
$\eps>0$ there is $v_0\in S(c)\cap\Sigma$ with $\|v_0-u_c\|_{\Ha}<\eps$,
$\Ea(v_0)<m_a(c)$ and $\Qa(v_0)<0$, whose solution blows up in finite time.
\end{theorem}

The closed formula \eqref{eq:levelformula} also shows how the threshold depends on
the coupling constant.

\begin{proposition}\label{prop:monotone}
Let $N\ge3$, $0<b<2$ and $0<\alpha<\frac{4-2b}{N-2}$. Then the map $a\mapsto\GN$
is strictly decreasing on $\big(-\frac{(N-2)^{2}}{4},\infty\big)$. If moreover
\eqref{eq:range} holds, then, for each fixed $c>0$, both $a\mapsto m_a(c)$ and
$a\mapsto\omega_c$ are strictly increasing.
\end{proposition}
\begin{remark}\rm \label{rmk:coupling}
Taking $a=0$ as a reference, Proposition \ref{prop:monotone} gives
$m_a(c)>m_0(c)$ for $a>0$ and $m_a(c)<m_0(c)$ for $\aH<a<0$, where $m_0(c)$ is
the normalized level of the inhomogeneous equation without potential. Thus, for
$b>0$, a repulsive potential strictly raises the threshold, and an attractive one
strictly lowers it. For $b=0$, the first assertion fails: by Remark
\ref{rmk:limiting}, $m_a(c)=m_0(c)$ for every $a\ge0$. The strict monotonicity
therefore rests on the attainment of $\GN$, which is where $b>0$ enters the proof
of Proposition \ref{prop:monotone}. Whether the set of data satisfying
\eqref{eq:below} and $\Qa(u_0)>0$ is monotone in $a$ is a different question
(Remark \ref{rmk:monotoneset}).
\end{remark}
One might expect the threshold to collapse when the potential approaches the
critical attractive strength $\aH$, where the form \eqref{eq:dHa} is no longer
comparable to $\|\nabla u\|_{L^2}^2$. This is not the case. Using a Hardy
inequality with a fractional remainder due to Frank \cite{Frank2009}, we show in
Proposition \ref{prop:hardyendpoint} that $\GN$ stays bounded as $a\downarrow\aH$.
Hence $m_a(c)$ decreases to a strictly positive limit.

\begin{remark}\rm\label{rmk:masscritical}
The exponents \eqref{eq:betasig} degenerate as $\alpha\downarrow\frac{4-2b}{N}$:
then $\beta\downarrow2$ and $\gamma\uparrow\infty$. In the mass-critical case,
the threshold is carried by the mass alone. We record the corresponding statement
in Proposition \ref{prop:masscritical}. It serves as a consistency check on
\eqref{eq:levelformula}, and it isolates the mechanism: for
$\alpha=\frac{4-2b}{N}$, the fibering map is a multiple of $t^2$, whose sign
changes at the critical mass $c^{*}=\big(\frac{\alpha+2}{2\GN}\big)^{1/\alpha}$.
\end{remark}

\begin{remark}\rm \label{rmk:general}
The variational results, namely Theorems \ref{thm:level} and \ref{thm:equiv} and
Propositions \ref{prop:monotone} and \ref{prop:hardyendpoint}, hold in the whole
range \eqref{eq:range} and in every dimension $N\ge3$. Their proofs are
self-contained; for the attainment of $\GN$, which is due to Campos and Guzm\'an
\cite[Lemma~3.2]{CamposGuzman2021}, we include a short proof (Proposition
\ref{prop:GN}). The dynamical results are stated under Assumption
\ref{ass:cauchy}. By Proposition \ref{prop:lwp}, this assumption holds for all
$N\ge3$, $0<b<2$, $a>\aH$ and $\alpha$ in \eqref{eq:range}. Only the scattering
statement in Theorem \ref{thm:dichotomy}(i) is restricted to \eqref{eq:dynrange},
because it relies on the scattering criterion of \cite{HLW2024}.

The Strichartz-space theory summarized in Table~\ref{tab_LWP} enters only through
that criterion. The second row of Table~\ref{tab_LWP} does not contain the
three-dimensional cubic model when $b>0$, since for $N=3$ it requires
$\alpha\le2-2b$. For the model \eqref{eq:dynrange}, the third row gives
$a>-\frac14+\frac{b^{2}}{9}$, which is \cite[Corollary~1.6(ii)]{CamposGuzman2021}.
The scattering theorem of \cite{HLW2024} is stated for every $a>-\frac14$, and we
use it in that form. We note, however, that the global solutions used in
\cite{HLW2024} are those of \cite{CamposGuzman2021}.
\end{remark}

\subsection{Main novelties and relation to previous work}\label{ssec:novelties}

We summarize what is new in this paper and how it compares with the two works
closest to it, the paper of Hajaiej, Luo and Wang \cite{HLW2024} and that of
Boulaaras, Ghanmi and the second author \cite{BGS2023}.

\smallskip\noindent\textbf{(N1) An explicit formula for the normalized level.}
Theorem \ref{thm:level}(i) gives $m_a(c)$ in closed form, in every dimension
$N\ge3$, in the whole range \eqref{eq:range} and for every $a>\aH$. In
\cite[Theorem~1.6]{HLW2024}, the level is shown to be positive and attained, in
the same range of parameters, by a minimization argument based on a
Brezis--Lieb type splitting. The maximum of the energy along each fibre is
computed in \cite[Lemma~5.2]{HLW2024} in terms of $A(u)$ and $B(u)$, but the
level $m_a(c)$ itself is not identified there. The paper \cite{BGS2023}
does not consider the problem with prescribed mass. The formula reduces the
constrained minimax problem to the sharp constant of \eqref{eq:GNintro}, and it
shows that $c\mapsto m_a(c)$ is a pure power. The attainment of $\GN$ and the
existence of a ground state of \eqref{eq:ellipticQ} are due to Campos and
Guzm\'an \cite[Lemma~3.2 and Proposition~1.1]{CamposGuzman2021}. Proposition
\ref{prop:GN} gives a short self-contained proof of the attainment of $\GN$, in
the normalization used here, for all $N\ge3$, $0<b<2$, $a>\aH$ and
$0<\alpha<\frac{4-2b}{N-2}$.

\smallskip\noindent\textbf{(N2) Normalized minimizers are exactly the ground
states.} Theorem \ref{thm:level}(ii)--(iii) identifies the minimizers of $m_a(c)$
with the Gagliardo--Nirenberg optimizers lying on $V(c)$. It proves that they are
ground states of \eqref{eq:stationary} at the explicit frequency $\omega_c$, and
that every ground state arises in this way. This answers the question left open
in \cite[Remark~1.7]{HLW2024}. A by-product is that the numbers
$\Mass(Q)^{\sigma}\Ea(Q)$ and $\|Q\|_{L^2}^{\sigma}\|Q\|_{\dHa}$ are the same for
every ground state $Q$ of \eqref{eq:ellipticQ}. Campos and Guzm\'an
\cite[Proposition~1.1]{CamposGuzman2021} state that all solutions of
\eqref{eq:ellipticQ} have the same mass, $\dHa$ norm and energy. Their argument,
based on the Pohozaev identities and on equality in \eqref{eq:GNintro}, applies to
optimizers. Theorem \ref{thm:level}(iii) shows that every ground state, in the
sense of least action, is such an optimizer. Hence the thresholds used in
\cite{CamposGuzman2021,HLW2024,BGS2023} are well defined without any uniqueness
statement. Moreover, by the converse part of Lemma \ref{lem:EL}, every optimizer
which solves \eqref{eq:ellipticQ} lies on the Pohozaev manifold. Hence the
functions used in these works are ground states, and their thresholds coincide
with ours. To our knowledge, uniqueness is open when $a\ne0$ and $b\ne0$. In
\cite{BGS2023}, the threshold is written in terms of one particular optimizer
$\varphi$, and the relation between optimizers and solutions of least action is
not addressed.

\smallskip\noindent\textbf{(N3) The normalized level as a sharp dynamical
threshold.} Theorem \ref{thm:dichotomy} describes the dynamics below the ground
state by a single number depending only on the mass, $m_a(c)$, and by the sign of
the functional $\Qa$ that defines $V(c)$. The alternative is exhaustive, because
$\Qa$ cannot vanish below the level. The level is sharp, since global
non-scattering standing waves exist at the level. By Theorem \ref{thm:equiv},
below the level, these conditions are equivalent to the scale-invariant
conditions of Holmer--Roudenko type. In this form, and for data of positive
energy, global existence and finite-time blow-up for data of finite variance or
radial data were proved by Campos and Guzm\'an
\cite[Theorem~1.3]{CamposGuzman2021}. Criteria for global
existence and blow-up of general solutions, which are not necessarily radial or of
finite variance, were obtained by An, Jang and Kim \cite{AJK2023}. On the blow-up
side, Theorem \ref{thm:dichotomy} adds the formulation through $m_a(c)$, the
virial bound \eqref{eq:blowup}, which is uniform in time, and an explicit upper
bound on the blow-up times (Remark \ref{rmk:comparison}). Both bounds are obtained
pointwise in time from \eqref{eq:keyid}, without any continuity argument. Theorem
\ref{thm:equiv} also identifies our conditions with the conserved conditions of
\cite{BGS2023}, with $\varphi$ in place of $Q$. This substitution is legitimate.
Indeed, the function $\varphi$ of \cite{BGS2023} is an optimizer of
\eqref{eq:GNintro} with $\Qa(\varphi)=0$. Hence it is a minimizer of
$m_a(\|\varphi\|_{L^2})$, and a ground state by Theorem
\ref{thm:level}(ii)--(iii). The paper \cite{BGS2023} contains no blow-up
statement.

\smallskip\noindent\textbf{(N4) Instability of normalized ground states.}
Theorem \ref{thm:instability} proves that every normalized ground state is
strongly unstable, with destabilizing data of finite variance on the same mass
sphere, given explicitly by the dilations $u_c^{t}$ with $t>1$. Proposition
\ref{prop:decay} gives the exponential decay needed for this, without any
regularity theory at the singularity. The instability by blow-up of the
ground-state standing waves of \eqref{eq:main}, understood as minimizers of the
action, was established by An, Mun and Kim \cite{AMK2024} in the intercritical
range, by a localized virial argument. Since Theorem \ref{thm:level}(iii) shows
that normalized ground states and ground states in the sense of least action
coincide, Theorem \ref{thm:instability} is a mass-constrained and quantitative
complement to \cite{AMK2024}. Indeed, the destabilizing data remain on $S(c)$,
satisfy the hypotheses of Theorem \ref{thm:dichotomy}(ii), and their blow-up times
obey the explicit bound of Remark \ref{rmk:comparison}. These questions are not considered in
\cite{BGS2023}, nor, to our knowledge, in \cite{HLW2024}.

\smallskip\noindent\textbf{(N5) Dependence of the threshold on the parameters.}
The closed formula gives the strict monotonicity of $m_a(c)$ and $\omega_c$ in the
coupling constant (Proposition \ref{prop:monotone}). It also gives a uniform
positive limit of $m_a(c)$ at the critical Hardy coupling $a\downarrow\aH$
(Proposition \ref{prop:hardyendpoint}), and it describes the mass-critical
borderline (Proposition \ref{prop:masscritical}).

\smallskip
\emph{Comparison with \cite{BGS2023}.} The two works are complementary. As far as
scattering is concerned, \cite{BGS2023} is more general than Theorem
\ref{thm:dichotomy}(i) with respect to $(\alpha,b)$. It covers every
intercritical power and every $0<b<2$, with a proof of its own. The scattering
statement of Theorem \ref{thm:dichotomy}(i) is not proved in the present paper. It
is taken from \cite{HLW2024}, and it is therefore restricted to $\alpha=2$ and
$0<b<1$. In \cite[Remark~1.2(i)]{HLW2024}, it is stated that the same arguments
apply to the nonlinearity $|x|^{-b}|u|^{\alpha}u$ with general $\alpha$ in
dimension three, without a detailed proof. On the other hand, through
\cite{HLW2024}, Theorem \ref{thm:dichotomy}(i) covers the attractive range
$-\frac14<a<0$. This range is excluded in \cite{BGS2023}, because the unweighted
dispersive estimate \eqref{eq:disp} is not available for $a<0$. The results listed in
(N1)--(N5) have no counterpart in \cite{BGS2023}, which quotes its variational
input from \cite{CamposGuzman2021} and treats neither blow-up nor instability. In
the other direction, Theorem \ref{thm:equiv} allows one to read the scattering
theorem \cite[Theorem~2.2]{BGS2023} as a statement about $m_a(c)$. For $N=3$,
$a>0$, $0<b<2$ and every $\alpha$ in \eqref{eq:range}, it asserts that data with
$\Ea(u_0)<m_a(c)$ and $\Qa(u_0)>0$ give global solutions that scatter. Together
with Theorem \ref{thm:dichotomy}(ii), this yields the full dichotomy for these
parameters. We do not restate it as a theorem, since its scattering part rests
entirely on \cite{BGS2023}. Finally, the formulation of \cite{BGS2023} through the
potential energy has a technical advantage, which we use in Section
\ref{sec:dynamics} (see the discussion after Proposition \ref{prop:scatcrit}): it
is inherited by spatial truncations.

\subsection{Strategy and organization}\label{ssec:strategy}

The proof of Theorem \ref{thm:level} rests on one identity. Along the dilations
$u^{t}(x)=t^{N/2}u(tx)$, which preserve the mass, the energy reads
$\Ea(u^{t})=\frac{t^{2}}{2}A-\frac{t^{\beta}}{\alpha+2}B$ with
$A=\|u\|_{\dHa}^{2}$ and $B=\int_{\R^N}\frac{|u|^{\alpha+2}}{|x|^b}\,dx$.
Eliminating $B$, we get
\begin{equation}\label{eq:keyid}
\Ea(u^{t})=\frac{\beta-2}{2\beta}\,t^{2}A+\frac1\beta\,\Qa(u^{t}),\qquad t>0 .
\end{equation}
We evaluate \eqref{eq:keyid} at the maximum point of $t\mapsto\Ea(u^t)$ and
express $B$ through the Gagliardo--Nirenberg quotient. Then the variable $A$ drops
out, and only the quotient and the mass remain. This produces
\eqref{eq:levelformula} and the characterization of minimizers. The ground-state
property is then obtained on the Nehari manifold, by an elementary optimization
in two variables that does not require the Pohozaev identity.

For Theorem \ref{thm:dichotomy}, the key observation is that the functional $\Qa$
defining $V(c)$ is one eighth of the second derivative of the variance,
$\frac{d^{2}}{dt^{2}}\int_{\R^N}|x|^{2}|u(t,x)|^{2}\,dx=8\Qa(u(t))$. The inverse-square potential
enters the virial identity with the same coefficient as the kinetic term,
precisely because it is scale critical. Identity \eqref{eq:keyid} then gives the
bound $\Qa(u)<\beta(\Ea(u)-m_a(c))$ whenever $\Qa(u)<0$ (Lemma
\ref{lem:coercivity}). This yields \eqref{eq:blowup} without any continuity or
bootstrap argument. In the opposite regime, it gives a uniform coercivity
estimate, which is what the scattering argument needs.

Section \ref{sec:prelim} collects the technical tools: the Sobolev spaces adapted
to $\Lc$, Strichartz estimates, a compactness lemma for the weighted nonlinearity,
the sharp Gagliardo--Nirenberg inequality with the existence of optimizers, the
Cauchy problem, the persistence of finite variance and the virial identity. Section
\ref{sec:variational} contains the variational analysis and the proofs of Theorems
\ref{thm:level} and \ref{thm:equiv} and of Proposition \ref{prop:monotone}.
Section \ref{sec:dynamics} proves Theorem \ref{thm:dichotomy}. Section
\ref{sec:instability} establishes the exponential decay of solutions of the
stationary problem and proves Theorem \ref{thm:instability}. Section
\ref{sec:remarks} discusses the mass-critical case and the behaviour of the
threshold at the Hardy endpoint. It ends with open problems, where we keep a
clear distinction between what is proved and what is only expected.

%%%%%%%%%%%%%%%%%%%%%%%%%%%%%%%%%%%%%%%%%%%%%%%%%%%%%%%%%%%%%%%%%%%%%%%%%%%%%%
\section{Preliminaries}\label{sec:prelim}
%%%%%%%%%%%%%%%%%%%%%%%%%%%%%%%%%%%%%%%%%%%%%%%%%%%%%%%%%%%%%%%%%%%%%%%%%%%%%%

Integrals are taken over $\R^{N}$ unless otherwise indicated. We write
$\|\cdot\|_{r}$ for the norm of $L^{r}(\R^{N})$. The notation $X\lesssim Y$ means
$X\le CY$ for a constant $C$ which may change from line to line and depends only
on the fixed parameters $N,a,b,\alpha$.
The functionals $A$, $B$, $\Mass$ and $\Sw$, the form
$\langle\cdot,\cdot\rangle_a$, the dilations $u^{t}$ and the scaling relations
\eqref{eq:scalingAB} were introduced in Subsection~\ref{ssec:notation}, together
with Definitions~\ref{def:weak} and~\ref{def:gs}.

We first check the identity \eqref{eq:sigma}, which will be used repeatedly to
pass between the mass-constrained and the scale-invariant formulations.

\begin{lemma}\label{lem:sigma}
Let $N\ge3$, $0<b<2$ and $\frac{4-2b}{N}<\alpha<\frac{4-2b}{N-2}$, and let
$s_c,\beta,\sigma_c,\gamma$ be as in \eqref{eq:sc} and \eqref{eq:betasig}. Then
$\beta>2$, $\sigma_c>0$, $\beta+\sigma_c=\alpha+2$, $0<s_c<1$, and
\begin{equation}\label{eq:sigmaid}
\frac\gamma2=\frac{\sigma_c}{\beta-2}=\frac{1-s_c}{s_c},
\qquad
\gamma+2=\frac{2\alpha}{\beta-2}=\frac{4\alpha}{N\alpha+2b-4}.
\end{equation}
\end{lemma}

\begin{proof}
The inequality $\alpha>\frac{4-2b}{N}$ is equivalent to $N\alpha+2b>4$, that is,
to $\beta>2$; the inequality $\alpha<\frac{4-2b}{N-2}$ is equivalent to
$(N-2)\alpha<4-2b$, that is, to $\sigma_c>0$. The same two inequalities say that
$s_c=\frac N2-\frac{2-b}{\alpha}$ lies in $(0,1)$, because $s_c>0$ reads
$N\alpha>2(2-b)$ and $s_c<1$ reads $(N-2)\alpha<2(2-b)$. Adding the definitions of
$\beta$ and $\sigma_c$ gives
$\beta+\sigma_c=\frac{N\alpha+2b+4-2b-(N-2)\alpha}{2}=\alpha+2$. For the first
identity in \eqref{eq:sigmaid} we multiply numerator and denominator of
$\frac{1-s_c}{s_c}$ by $2\alpha$ and obtain
\begin{equation*}
\frac{1-s_c}{s_c}=\frac{2\alpha-N\alpha+2(2-b)}{N\alpha-2(2-b)}
=\frac{4-2b-(N-2)\alpha}{N\alpha+2b-4}=\frac{\sigma_c}{\beta-2},
\end{equation*}
which equals $\frac\gamma2$ by \eqref{eq:betasig}. Finally
$\gamma+2=\frac{2\sigma_c+2(\beta-2)}{\beta-2}=\frac{2(\alpha+2)-4}{\beta-2}
=\frac{2\alpha}{\beta-2}$.
\end{proof}

\subsection{\texorpdfstring{Function spaces adapted to $\Lc$}{Function spaces adapted to the operator}}\label{ssec:spaces}

For $1<r<\infty$ we denote by $\dot H^{1,r}_a$ and $H^{1,r}_a$ the homogeneous and
inhomogeneous Sobolev spaces associated with $\Lc$, with norms
$\|u\|_{\dot H^{1,r}_a}=\|\Lc^{1/2}u\|_{r}$ and
$\|u\|_{H^{1,r}_a}=\|(1+\Lc)^{1/2}u\|_{r}$. For $r=2$, we drop the second index,
in agreement with \eqref{eq:dHa}. The comparison with the classical spaces is
governed by
\begin{equation}\label{eq:rho}
\rho:=\frac{N-2}{2}-\sqrt{\frac{(N-2)^{2}}{4}+a}\;,
\end{equation}
which is negative for $a>0$ and lies in $\big(0,\frac{N-2}{2}\big)$ for
$\aH<a<0$. It describes the behaviour at the origin of the zero modes of $\Lc$.
The following result is due to Killip, Miao, Visan, Zhang and Zheng
\cite{KMVZZ2018}.

\begin{proposition}[Equivalence of Sobolev spaces, \cite{KMVZZ2018}]\label{prop:equivsobolev}
Let $N\ge3$, $a>-\frac{(N-2)^{2}}{4}$ and $0<s<2$.
\begin{enumerate}
    \item 
Let $1<r<\infty$ and $f\in C^{\infty}_{c}(\R^{N}\setminus\{0\})$. 
\begin{enumerate}
    \item 
If $\frac{s+\rho}{N}<\frac1r<\min\{1,\frac{N-\rho}{N}\}$, then
$\big\||\nabla|^{s}f\big\|_{r}\lesssim\big\|\Lc^{s/2}f\big\|_{r}$.

\item If
$\max\{\frac sN,\frac\rho N\}<\frac1r<\min\{1,\frac{N-\rho}{N}\}$, the reverse
inequality $\big\|\Lc^{s/2}f\big\|_{r}\lesssim\big\||\nabla|^{s}f\big\|_{r}$ holds.
\end{enumerate}
\item 
In the Hilbert case, we have, for every $u\in H^1(\R^N)$,
\begin{equation}\label{eq:equivnorm}
\lambda_a\,\|\nabla u\|_{2}^{2}\le\|u\|_{\dHa}^{2}
\le\Lambda_a\,\|\nabla u\|_{2}^{2},\end{equation}

where
\begin{equation*}
\lambda_a:=\min\Big\{1,1+\frac{4a}{(N-2)^{2}}\Big\},\qquad
\Lambda_a:=\max\Big\{1,1+\frac{4a}{(N-2)^{2}}\Big\}.
\end{equation*}
\end{enumerate}
\end{proposition}

The two-sided bound \eqref{eq:equivnorm} is an immediate consequence of
\eqref{eq:hardy}: if $a\ge0$, then
$\|\nabla u\|_{2}^{2}\le\|u\|_{\dHa}^{2}\le(1+\frac{4a}{(N-2)^{2}})\|\nabla u\|_{2}^{2}$,
and if $a<0$ the two inequalities are reversed. In particular, $\Ha(\R^{N})$ and
$H^{1}(\R^{N})$ coincide as sets, with equivalent norms, and $\lambda_a>0$
exactly when $a>\aH$. We use this without further comment. The $L^r$ statements
above are quoted from \cite{KMVZZ2018}. They enter only through the
Strichartz-space theory of \cite{CamposGuzman2021} recalled in
Table~\ref{tab_LWP}. The Hilbert-space bound \eqref{eq:equivnorm} is the only part
used in our own proofs.

\subsection{Dispersive and Strichartz estimates}\label{ssec:strichartz}

In three dimensions and for $a\ge0$, Fanelli, Felli, Fontelos and Primo
\cite{FFFP2013} proved the dispersive estimate
\begin{equation}\label{eq:disp}
\|e^{-it\Lc}f\|_{\infty}\lesssim|t|^{-3/2}\|f\|_{1}.
\end{equation}
For $-\frac14<a<0$, they obtained the weighted version
\begin{equation}\label{eq:dispweighted}
\big\|(1+|x|^{-\rho})^{-1}e^{-it\Lc}f\big\|_{\infty}
\lesssim\frac{1+|t|^{\rho}}{|t|^{3/2}}\big\|(1+|x|^{-\rho})f\big\|_{1},
\end{equation}
with $\rho$ as in \eqref{eq:rho}. Zheng \cite{Zheng2018} proved the analogues of
\eqref{eq:disp} and \eqref{eq:dispweighted} in every dimension $N\ge3$ for radial
data; see also \cite[Lemma~2.4]{CamposGuzman2021}. The loss in
\eqref{eq:dispweighted} is one of the technical difficulties handled in
\cite{HLW2024}. These dispersive estimates are recorded for context only, and
they are not used in our proofs.

We call a pair $(q,r)$ \emph{admissible} if $2\le q\le\infty$ and
$\frac2q+\frac Nr=\frac N2$. The following result is due to Burq, Planchon,
Stalker and Tahvildar-Zadeh \cite{BPST2003} for non-endpoint pairs. The endpoint
$q=2$ was treated by Bouclet and Mizutani \cite{BoucletMizutani2018}; see also
\cite[Lemma~2.5]{CamposGuzman2021}. Away from the endpoint, the inhomogeneous
estimate follows from the homogeneous one by the Christ--Kiselev lemma. In our own arguments, Proposition \ref{prop:strichartz}
is used only once, with $F=0$ and the non-endpoint admissible pair
$(4,\frac{2N}{N-1})$, in the proof of the sharpness part of Theorem
\ref{thm:dichotomy}.

\begin{proposition}[Strichartz estimates, \cite{BPST2003,BoucletMizutani2018}]\label{prop:strichartz}
Let $N\ge3$ and $a>-\frac{(N-2)^{2}}{4}$. If $u$ solves
$(i\partial_t-\Lc)u=F$ on an interval $I\ni t_0$, then, for all admissible pairs
$(q,r)$ and $(\tilde q,\tilde r)$,
\begin{equation}\label{eq:strichartz}
\|u\|_{L^{q}_{t}L^{r}_{x}(I\times\R^{N})}
\lesssim\|u(t_0)\|_{2}+\|F\|_{L^{\tilde q'}_{t}L^{\tilde r'}_{x}(I\times\R^{N})} .
\end{equation}
\end{proposition}

\subsection{The weighted nonlinearity and the sharp Gagliardo--Nirenberg inequality}\label{ssec:GN}

The variational analysis rests on two properties of the functional $B$: it is
controlled by the $H^1$ norm, and it is weakly continuous. The weight $|x|^{-b}$
plays its role in the second property. For $b=0$, this property fails, since mass
can escape to infinity.

\begin{lemma}\label{lem:weighted}
Let $N\ge3$, $0<b<2$ and $0<\alpha<\frac{4-2b}{N-2}$.
\begin{enumerate}
\item[(a)] There is $C>0$ such that $B(u)\le C\|u\|_{H^1}^{\alpha+2}$ and
\begin{equation*}
\big|B(u)-B(v)\big|\le C\big(\|u\|_{H^1}+\|v\|_{H^1}\big)^{\alpha+1}\|u-v\|_{H^1}
\end{equation*}
for all $u,v\in H^1(\R^N)$. Moreover, $B$ is of class $C^1$ on $H^1(\R^N)$, and
\begin{equation*}
B'(u)h=(\alpha+2)\Ree\int|x|^{-b}|u|^{\alpha}u\,\bar h\dd,
\qquad u,h\in H^1(\R^N).
\end{equation*}
\item[(b)] If $u_n\rightharpoonup u$ weakly in $H^1(\R^N)$, then $B(u_n)\to B(u)$.
\end{enumerate}
\end{lemma}

\begin{proof}
Since $\alpha(N-2)<4-2b$, we have $(\alpha+2)\frac{N}{N-b}<\frac{2N}{N-2}$. Indeed,
this inequality is equivalent to $(\alpha+2)(N-2)<2(N-b)$, that is, to
$\alpha(N-2)<4-2b$. We may therefore fix $r\in(1,\frac Nb)$ so close to $\frac Nb$
that $p:=(\alpha+2)r'<\frac{2N}{N-2}$, where $r'$ is the conjugate exponent. Note
that $p>2$. Then $|x|^{-b}\in L^{r}(B_1)$, and H\"older's inequality followed by
the Sobolev embedding gives
\begin{equation*}
\int_{|x|<1}\frac{|w|^{\alpha+2}}{|x|^{b}}\dd\le\big\||x|^{-b}\big\|_{L^{r}(B_1)}
\|w\|_{L^{p}(B_1)}^{\alpha+2}\lesssim\|w\|_{H^1}^{\alpha+2},
\qquad
\int_{|x|\ge1}\frac{|w|^{\alpha+2}}{|x|^{b}}\dd\le\|w\|_{\alpha+2}^{\alpha+2}
\lesssim\|w\|_{H^1}^{\alpha+2},
\end{equation*}
where we used $2<\alpha+2<\frac{2N}{N-2}$ in the second estimate. We deduce
that $\|w\|_{L^{\alpha+2}(|x|^{-b}dx)}\lesssim\|w\|_{H^1}$, which is the first
claim in (a). For the second, we use
$\big||u|^{\alpha+2}-|v|^{\alpha+2}\big|\le(\alpha+2)(|u|^{\alpha+1}+|v|^{\alpha+1})|u-v|$
and H\"older's inequality with exponents $\frac{\alpha+2}{\alpha+1}$ and $\alpha+2$
for the measure $|x|^{-b}dx$. We obtain
$|B(u)-B(v)|\le(\alpha+2)\big(B(u)^{\frac{\alpha+1}{\alpha+2}}+B(v)^{\frac{\alpha+1}{\alpha+2}}\big)
\|u-v\|_{L^{\alpha+2}(|x|^{-b}dx)}$, and the claim follows from the first bound.
The differentiability of $B$ is a routine consequence of the same estimates
applied to the map $z\mapsto|z|^{\alpha+2}$, whose derivative is locally
H\"older continuous; see \cite[Chapter~3]{Cazenave2003} or \cite{Farah2016}.

For (b), let $u_n\rightharpoonup u$ in $H^1$. Then
$M_0:=\sup_n\|u_n\|_{H^1}+\|u\|_{H^1}<\infty$. Put $w_n:=u_n-u$. For $R\ge1$,
applying the argument above on $B_R$ and on its complement, we get
\begin{equation*}
\int\frac{|w_n|^{\alpha+2}}{|x|^{b}}\dd
\le\big\||x|^{-b}\big\|_{L^{r}(B_R)}\|w_n\|_{L^{p}(B_R)}^{\alpha+2}
+R^{-b}\|w_n\|_{\alpha+2}^{\alpha+2}
\le C_R\|w_n\|_{L^{p}(B_R)}^{\alpha+2}+CR^{-b}M_0^{\alpha+2}.
\end{equation*}
Given $\eps>0$ we choose $R$ with $CR^{-b}M_0^{\alpha+2}<\eps$. Since
$p<\frac{2N}{N-2}$, the Rellich--Kondrachov theorem gives $w_n\to0$ in
$L^{p}(B_R)$, so the first term tends to zero. Hence
$\|w_n\|_{L^{\alpha+2}(|x|^{-b}dx)}\to0$, and the difference estimate used in (a)
yields $B(u_n)\to B(u)$.
\end{proof}

We can now state the sharp inequality. The inequality and the existence of an
optimizer are due to Campos and Guzm\'an \cite[Lemma~3.2]{CamposGuzman2021}. For
$a=0$, see Farah \cite{Farah2016} and Guzm\'an \cite{Guzman2017}, and for
fractional operators, see Bhimani, Hajaiej, Haque and Luo \cite{BHHL2023}. We
include a short proof for completeness, since the normalization
\eqref{eq:GNquotient} and the invariances of $J$ are used throughout.

\begin{proposition}\label{prop:GN}
Let $N\ge 3$, $a>-\frac{(N-2)^2}{4}$, $0<b<2$ and $0<\alpha<\frac{4-2b}{N-2}$,
and let $\beta$ and $\sigma_c$ be defined as in \eqref{eq:betasig}. Define the
sharp Gagliardo--Nirenberg constant by
\begin{equation}\label{eq:GNquotient}
\GN:=\sup_{u\in\Ha(\mathbb{R}^N)\setminus\{0\}}J(u),
\qquad
J(u):=\frac{B(u)}
{A(u)^{\beta/2}\|u\|_{2}^{\sigma_c}},
\end{equation}
where $A(u)$ and $B(u)$ are given by \eqref{eq:ABm}. Then $0<\GN<\infty$, and the
supremum in \eqref{eq:GNquotient} is attained. Consequently, the sharp
Gagliardo--Nirenberg inequality
\begin{equation}\label{eq:GN}
B(u)\le \GN\,A(u)^{\beta/2}\|u\|_{2}^{\sigma_c},
\qquad u\in\Ha(\mathbb{R}^N),
\end{equation}
holds, and $\GN$ is the optimal constant.
\end{proposition}

\begin{proof}
Note first that $\beta+\sigma_c=\alpha+2$, $\beta>0$ and $\sigma_c>0$ for every
$\alpha$ in the stated range. By Lemma \ref{lem:weighted}(a),
$B(w)\le C\big(\|\nabla w\|_2^{2}+\|w\|_2^{2}\big)^{\frac{\alpha+2}{2}}$ for all
$w\in H^1$. Given $u\ne0$, we apply this to $w=u^{t}$ and use
\eqref{eq:scalingAB} together with $\|\nabla u^t\|_2^2=t^2\|\nabla u\|_2^2$. We get
\begin{equation*}
t^{\beta}B(u)\le C\big(t^{2}\|\nabla u\|_2^{2}+\|u\|_2^{2}\big)^{\frac{\alpha+2}{2}}.
\end{equation*}
The choice $t^{2}=\|u\|_2^{2}/\|\nabla u\|_2^{2}$ gives
$B(u)\le 2^{\frac{\alpha+2}{2}}C\,\|\nabla u\|_2^{\beta}\|u\|_2^{\sigma_c}$, and
\eqref{eq:equivnorm} turns this into \eqref{eq:GN} with
$\GN\le2^{\frac{\alpha+2}{2}}C\lambda_a^{-\beta/2}<\infty$. Positivity is obvious.

For the existence of an optimizer, we use the two-parameter invariance of $J$. For
$\theta>0$ and $t>0$, we have $J(\theta u^{t})=J(u)$, because the numerator and the
denominator of $J$ are both homogeneous of degree $\alpha+2$ in $u$, and because
$t^{\beta}B/(t^{2}A)^{\beta/2}=B/A^{\beta/2}$ by \eqref{eq:scalingAB}. Let $(u_n)$
be a maximizing sequence. Replacing $u_n$ by $\theta_nu_n^{t_n}$ with
$\theta_n=\|u_n\|_2^{-1}$ and $t_n=\big(\Mass(u_n)/A(u_n)\big)^{1/2}$, we may
assume $\|u_n\|_2=1$ and $A(u_n)=1$. Then $(u_n)$ is bounded in $H^1$, and after
extraction $u_n\rightharpoonup u$ weakly in $H^1$. By Lemma
\ref{lem:weighted}(b), we have $B(u)=\lim B(u_n)=\lim J(u_n)=\GN>0$, so $u\ne0$. The
functional $A$ is a continuous nonnegative quadratic form on $H^1$, hence convex
and weakly lower semicontinuous, so $A(u)\le1$; likewise $\|u\|_2\le1$.
Consequently $J(u)\ge B(u)=\GN$, and $u$ is an optimizer.
\end{proof}

Two features of $J$ appeared in this proof and will be used constantly: it is
homogeneous of degree zero, and it is invariant under the dilations
$u\mapsto u^{t}$. Together they say that the optimization problem
\eqref{eq:GNquotient} may be posed on any sphere $S(c)$.

\begin{lemma}\label{lem:GNrestrict}
For every $c>0$, we have
\begin{equation*}
\GN=\sup_{u\in S(c)}J(u).
\end{equation*}
Moreover, a nonzero function $u\in\Ha(\mathbb{R}^N)$ is an optimizer for \eqref{eq:GN} if and only if $\theta u^{t}$ is an optimizer for every $\theta>0$ and $t>0$.
\end{lemma}

\begin{proof}
Given $u\in\Ha\setminus\{0\}$, set $\theta:=c/\|u\|_{2}$. Then $\theta u\in S(c)$ and
$J(\theta u)=J(u)$. This proves the first assertion. The second follows from
the invariance $J(\theta u^{t})=J(u)$ noted above.
\end{proof}

\begin{remark}\rm \label{rmk:attained}
For $a=0$, the existence of optimizers is due to Farah \cite{Farah2016}. For $b=0$, the situation is
different. When $a>0$, the best constant equals the one for $a=0$ and is not
attained, while for $\aH<a<0$ it is attained \cite{KMVZ2017,MNN2021}. As soon as
$b>0$, the decay of the weight restores compactness, and the constant is attained
for every $a>\aH$, including the repulsive case
(\cite[Lemma~3.2]{CamposGuzman2021} and Proposition \ref{prop:GN}). This is the
only point where compactness enters the paper.
\end{remark}

\subsection{The Cauchy problem}\label{ssec:cauchy}

We now explain why the properties collected in Assumption~\ref{ass:cauchy}
hold.

The local theory in $\Ha(\R^N)$ was established by Suzuki \cite{Suzuki2016}, by
the energy method of Okazawa, Suzuki and Yokota \cite{Suzuki2012};
see also \cite[Section~1]{CamposGuzman2021}. A Strichartz-space theory, obtained
by a contraction argument as for the classical equation
\cite{Cazenave2003}, was developed in
\cite[Section~4]{CamposGuzman2021} under the restrictions recalled in
Table~\ref{tab_LWP}. For $a=0$, see
\cite{GenoudStuart2008,Guzman2017}, and for $b=0$, see
\cite{KMVZ2017,Dinh2021}. The persistence of finite variance follows from the next
lemma, which requires only the regularity of energy solutions.

\begin{lemma}[Persistence of finite variance]\label{lem:variance}
Let $N\ge3$, $a>\aH$, $0<b<2$ and $0<\alpha<\frac{4-2b}{N-2}$, let $I\ni0$ be an
interval, and let $u\in C(I;\Ha(\R^N))$ solve \eqref{eq:main} in
$H^{-1}(\R^N)$ with $u(0)=u_0\in\Sigma$. Then $u\in C^1(I;H^{-1}(\R^N))$,
$|x|u\in C(I;L^2(\R^N))$, and $V(t):=\int|x|^2|u(t,x)|^2\dd$ is of class $C^1$ on
$I$ with $V'(t)=4\Imm\int\bar u\,x\cdot\nabla u\dd$.
\end{lemma}

\begin{proof}
The operator $\Lc$ is bounded from $\Ha$ to $H^{-1}$. By Lemma
\ref{lem:weighted}(a), the map
$u\mapsto|x|^{-b}|u|^{\alpha}u=\frac{1}{\alpha+2}B'(u)$ is continuous from $H^1$
to $H^{-1}$. The equation therefore gives
$\partial_tu=-i\big(\Lc u-|x|^{-b}|u|^\alpha u\big)\in C(I;H^{-1})$, so that
$u\in C^1(I;H^{-1})$.

For $\eps>0$, put $\psi_\eps(x):=|x|^2e^{-\eps|x|^2}$. This is a smooth bounded
function with bounded gradient. Since
$\nabla\psi_\eps=2x(1-\eps|x|^2)e^{-\eps|x|^2}$ and
$\sup_{s\ge0}|1-s|e^{-s/2}=1$, we have
\begin{equation}\label{eq:psieps}
|\nabla\psi_\eps|\le2\psi_\eps^{1/2},
\end{equation}
uniformly in $\eps$. Multiplication by $\psi_\eps$ is bounded on
$H^1$. Let $f_\eps(t):=\int\psi_\eps|u(t)|^2\dd$. A computation of difference
quotients shows that $f_\eps$ is of class $C^1$ on $I$, with
\begin{equation*}
f_\eps'(t)=2\Ree\langle\partial_tu,\psi_\eps u\rangle_{H^{-1},H^1}.
\end{equation*}
We have
\begin{equation*}
\langle\Lc u,\psi_\eps u\rangle=\int\psi_\eps|\nabla u|^2\dd
+a\int\psi_\eps\frac{|u|^2}{|x|^{2}}\dd+\int\bar u\,\nabla\psi_\eps\cdot\nabla u\dd ,
\end{equation*}
and $\langle|x|^{-b}|u|^\alpha u,\psi_\eps u\rangle=\int\psi_\eps|x|^{-b}|u|^{\alpha+2}\dd$
is real. Hence $f_\eps'(t)=2\Imm\int\bar u\,\nabla\psi_\eps\cdot\nabla u\dd$, and
by \eqref{eq:psieps} and the Cauchy--Schwarz inequality, we get
$|f_\eps'(t)|\le4f_\eps(t)^{1/2}\|\nabla u(t)\|_2$. For $\delta>0$, this gives
$\big|\frac{d}{dt}(f_\eps+\delta)^{1/2}\big|\le2\|\nabla u(t)\|_2$. Integrating and
letting $\delta\downarrow0$, we obtain
\begin{equation*}
f_\eps(t)^{1/2}\le\||x|u_0\|_2+2\Big|\int_0^t\|\nabla u(s)\|_2\,ds\Big| .
\end{equation*}
Since $\psi_\eps\uparrow|x|^2$ as $\eps\downarrow0$, the monotone convergence
theorem gives $u(t)\in\Sigma$, with a bound on $\||x|u(t)\|_2$ that is locally
uniform in $t$.

Next, we have
$f_\eps(t)=f_\eps(0)+2\int_0^t\Imm\int\bar u\,\nabla\psi_\eps\cdot\nabla u\dd\,ds$.
As $\eps\downarrow0$, we have $\nabla\psi_\eps\to2x$ pointwise and
$|\nabla\psi_\eps|\le2|x|$. Moreover, $|x||u||\nabla u|$ is integrable on
$[0,t]\times\R^N$ by the bound just proved. The dominated convergence theorem
gives
\begin{equation*}
V(t)=V(0)+4\int_0^t\Imm\int\bar u\,x\cdot\nabla u\dd\,ds,
\end{equation*}
so $V$ is continuous. For $\varphi\in C_c^\infty(\R^N)$, the map
$t\mapsto\int|x|u(t)\bar\varphi\dd$ is continuous, and $\||x|u(t)\|_2$ is locally
bounded. Hence $|x|u$ is weakly continuous with values in $L^2$. Together with the
continuity of $V(t)=\||x|u(t)\|_2^2$, this gives $|x|u\in C(I;L^2)$. The map
$s\mapsto\Imm\int\bar u\,x\cdot\nabla u\dd$ is then continuous, and we infer that
$V\in C^1(I)$ with the stated derivative.
\end{proof}

\begin{proposition}\label{prop:lwp}
Let $N\ge3$, $0<b<2$, $a>\aH$, and let $\alpha$ satisfy \eqref{eq:range}. Then
Assumption \ref{ass:cauchy} holds.
\end{proposition}

\begin{proof}
The first part of Assumption \ref{ass:cauchy} is due to Suzuki
\cite{Suzuki2016,Suzuki2012}, as recalled in
\cite{CamposGuzman2021}. Since the mass is conserved, the blow-up of the $\Ha$
norm is equivalent to the blow-up of the $\dHa$ norm. The persistence of finite
variance and the continuity of $t\mapsto\||x|u(t)\|_2$ follow from Lemma
\ref{lem:variance}. The virial identity \eqref{eq:virial1} is Lemma
\ref{lem:virial} below.
\end{proof}

\subsection{Virial identities}\label{ssec:virial}

The last ingredient is the virial identity. It explains why the Pohozaev
functional $\Qa$, defined in \eqref{eq:Vc} as the derivative at $t=1$ of the
fibering map, appears in a dynamical statement. The same quantity governs the
convexity of the variance.

\begin{lemma}\label{lem:virial}
Let $N\ge3$, $0<b<2$, $a>\aH$ and $0<\alpha<\frac{4-2b}{N-2}$. Let $u_0\in\Sigma$,
and let $u$ be the maximal solution of \eqref{eq:main} given by the first part of
Assumption \ref{ass:cauchy}. Then $V(t):=\int|x|^{2}|u(t,x)|^{2}\dd$ is of class
$C^{2}$ on $(-T_-,T_+)$, and
\begin{equation}\label{eq:virial1}
V'(t)=4\Imm\int\bar u\,x\cdot\nabla u\dd,
\qquad
V''(t)=8\,\Qa\big(u(t)\big).
\end{equation}
\end{lemma}

\begin{proof}
The first identity in \eqref{eq:virial1} is Lemma \ref{lem:variance}. The second
one is stated for \eqref{eq:main}, with a general weight, in
\cite[Lemma~3.4]{CamposGuzman2021}. We give the computation for the weight
$|x|^2$, because the exact coefficients matter here (see Remark
\ref{rmk:virialtypo}). We write \eqref{eq:main} as $i\partial_tu+\Delta u=G$ with
$G:=\frac{a}{|x|^{2}}u-|x|^{-b}|u|^{\alpha}u$, and we follow the classical
computation \cite[Chapter~6]{Cazenave2003}. Since $\bar uG$ is real, we have
$\partial_t|u|^2=-2\,\mathrm{div}\,\Imm(\bar u\nabla u)$. For a smooth real weight
$w$, we have
\begin{equation}\label{eq:virialgeneral}
\frac{d}{dt}\Big(2\Imm\int\bar u\,\nabla w\cdot\nabla u\dd\Big)
=\int\Big(-\Delta\Delta w\,|u|^{2}+4\Ree(\partial_j\bar u\,\partial_k u)\,\partial_{jk}w
+2\nabla w\cdot\{G,u\}_P\Big)\dd,
\end{equation}
with $\{f,g\}_P:=\Ree(f\nabla\bar g-g\nabla\bar f)$ and summation over repeated
indices. For $w=|x|^{2}$, we have $\Delta\Delta w=0$, $\partial_{jk}w=2\delta_{jk}$
and $\nabla w=2x$, so the first two terms give $8\|\nabla u\|_{2}^{2}$. For the
potential, we have $\{\frac{a}{|x|^{2}}u,u\}_P=-a|u|^{2}\nabla(|x|^{-2})=2a|u|^{2}\frac{x}{|x|^{4}}$,
hence $2\nabla w\cdot\{\frac{a}{|x|^{2}}u,u\}_P=8a\frac{|u|^{2}}{|x|^{2}}$. Thus the
inverse-square potential contributes with the same coefficient as the kinetic
term, which is the analytic expression of its scale invariance. For the
nonlinear part, we write $W:=|x|^{-b}$ and $h:=W|u|^{\alpha+2}$. Then
$\{-W|u|^{\alpha}u,u\}_P=|u|^{2}\nabla\big(W|u|^{\alpha}\big)$. Separating the
derivative that falls on $W$ from the one that falls on $|u|$, we find
\begin{equation*}
2\nabla w\cdot\big\{-W|u|^{\alpha}u,\,u\big\}_P
=\frac{4\alpha}{\alpha+2}\,x\cdot\nabla h
+\frac{8}{\alpha+2}\,|u|^{\alpha+2}\,x\cdot\nabla W .
\end{equation*}
Since $x\cdot\nabla W=-bW$ and $\int x\cdot\nabla h=-N\int h$, the integral of the
last expression equals
\begin{equation*}
-\Big(\frac{4N\alpha}{\alpha+2}+\frac{8b}{\alpha+2}\Big)\int h\dd
=-\frac{4(N\alpha+2b)}{\alpha+2}\int\frac{|u|^{\alpha+2}}{|x|^{b}}\dd .
\end{equation*}
Collecting the three contributions, we obtain
\begin{equation*}
V''(t)=8\|u(t)\|_{\dHa}^{2}-\frac{4(N\alpha+2b)}{\alpha+2}\int\frac{|u(t)|^{\alpha+2}}{|x|^{b}}\dd
=8\Qa(u(t))
\end{equation*}
by \eqref{eq:Qdef}. This computation is first performed for regular
solutions. It extends to the solutions considered here by the regularization
argument of \cite[Section~6.5]{Cazenave2003}, combined with the continuous
dependence on the data and with Lemma \ref{lem:variance}; see also
\cite[Lemma~3.4]{CamposGuzman2021} and \cite{Dinh2021}. The singular terms cause no
difficulty, since $|x|^{-2}|u|^2$ and $|x|^{-b}|u|^{\alpha+2}$ are integrable by
\eqref{eq:hardy} and Lemma \ref{lem:weighted}.
\end{proof}

\begin{remark}\rm \label{rmk:virialtypo}
With the defocusing sign the same computation gives
\begin{equation*}
V''(t)=8\|u(t)\|_{\dHa}^{2}+\frac{4(N\alpha+2b)}{\alpha+2}\int_{\R^N}\frac{|u|^{\alpha+2}}{|x|^b}\,dx.
\end{equation*}
For $N=3$ and $\alpha=2$, the coefficient of the nonlinear term is $2(3+b)$. In
\cite[Lemma~4.1]{HLW2024}, it is written as $8$. The same coefficient $8$ appears
in the interior term of the truncated identity \cite[Lemma~4.2]{HLW2024}, whose
exterior term carries the coefficient $2(2+b)$. In the defocusing case, only the
sign of the nonlinear term matters. In the focusing case, the interior term is
combined in \cite[(4.9)]{HLW2024} with the coercivity estimate
\cite[Proposition~2.5(ii)]{HLW2024}, which is written there as
$A(u)-B(u)\ge c\,B(u)$. This estimate cannot hold for small $\delta$ in the
hypothesis of \cite[Proposition~2.5]{HLW2024}. Indeed, the function $Q_{a,b}$ used
in \cite{HLW2024} satisfies $\Qa(Q_{a,b})=0$ by Lemma \ref{lem:EL}, that is,
$A(Q_{a,b})=\frac{3+b}{4}B(Q_{a,b})<B(Q_{a,b})$. Consequently, for
$t_*<t<1$ with $t_*:=\big(\frac{3+b}{4}\big)^{\frac{1}{1+b}}$, the dilate
$Q_{a,b}^{t}$ has the mass of $Q_{a,b}$, a strictly smaller energy and a strictly
smaller $\dHa$ norm, whereas $A(Q_{a,b}^{t})-B(Q_{a,b}^{t})=t^{2}A(Q_{a,b})-t^{3+b}B(Q_{a,b})<0$.
With the exact coefficient, the interior term is $8\Qa$ evaluated on the
truncated solution, up to the usual error terms. The coercivity which is actually
needed is therefore $\Qa(\chi u(t))\gtrsim\|\chi u(t)\|_{\dHa}^{2}$, and we prove it
below the level in Section \ref{sec:dynamics}. Here, too, the exact coefficient
matters, since it identifies $V''$ with $8\Qa$.
\end{remark}

%%%%%%%%%%%%%%%%%%%%%%%%%%%%%%%%%%%%%%%%%%%%%%%%%%%%%%%%%%%%%%%%%%%%%%%%%%%%%%
\section{The normalized level and its minimizers}\label{sec:variational}
%%%%%%%%%%%%%%%%%%%%%%%%%%%%%%%%%%%%%%%%%%%%%%%%%%%%%%%%%%%%%%%%%%%%%%%%%%%%%%

Throughout this section, $N\ge3$, $a>-\frac{(N-2)^{2}}{4}$, $0<b<2$, and
\eqref{eq:range} holds. Everything rests on an elementary but complete
description of the fibering map.

\begin{lemma}\label{lem:fibering}
Let $u\in\Ha(\R^{N})\setminus\{0\}$ and set $h_u(t):=\Ea(u^{t})$ for $t>0$. Then
\begin{equation}\label{eq:hu}
h_u(t)=\frac{t^{2}}{2}A(u)-\frac{t^{\beta}}{\alpha+2}B(u),
\qquad
t\,h_u'(t)=\Qa(u^{t})=t^{2}A(u)-\frac{\beta t^{\beta}}{\alpha+2}B(u),
\end{equation}
and
\begin{equation}\label{eq:master}
\Ea(u^{t})=\frac{\beta-2}{2\beta}\,t^{2}A(u)+\frac1\beta\,\Qa(u^{t}),
\qquad t>0.
\end{equation}
The map $h_u$ has exactly one critical point,
\begin{equation}\label{eq:tu}
t_u=\Big(\frac{(\alpha+2)A(u)}{\beta B(u)}\Big)^{\frac{1}{\beta-2}}\in(0,\infty),
\end{equation}
which is a strict global maximum. The map $h_u$ increases strictly on $(0,t_u)$
and decreases strictly on $(t_u,\infty)$, with $h_u(t)\to0$ as $t\downarrow0$ and
$h_u(t)\to-\infty$ as $t\to\infty$. Equivalently, $\Qa(u^{t})>0$ for $t<t_u$,
$\Qa(u^{t_u})=0$ and $\Qa(u^{t})<0$ for $t>t_u$. In particular $t_u$ is the only
$t>0$ with $u^{t}\in V(\|u\|_{2})$, and $t_u<1$ if and only if $\Qa(u)<0$, while
$t_u>1$ if and only if $\Qa(u)>0$. Finally,
\begin{equation}\label{eq:maxvalue}
\max_{t>0}\Ea(u^{t})=h_u(t_u)
=K_{\alpha,N,b}\,\frac{A(u)^{\frac{\beta}{\beta-2}}}{B(u)^{\frac{2}{\beta-2}}}
=K_{\alpha,N,b}\,J(u)^{-\frac{2}{\beta-2}}\,\|u\|_{2}^{-\gamma},
\end{equation}
with $J$ as in \eqref{eq:GNquotient} and $K_{\alpha,N,b}$ as in \eqref{eq:K}.
\end{lemma}

\begin{proof}
The formula for $h_u$ follows from \eqref{eq:EQAB} and \eqref{eq:scalingAB}.
Differentiating it, we get $t h_u'(t)=t^{2}A-\frac{\beta t^{\beta}}{\alpha+2}B$,
which is $\Qa(u^{t})$ by the same two relations. Identity \eqref{eq:master} is a
matter of elimination. Subtracting $\frac1\beta\Qa(u^{t})$ from $\Ea(u^{t})$
removes the term containing $B(u)$ and leaves
$\frac{t^{2}}{2}A-\frac{t^{2}}{\beta}A=\frac{\beta-2}{2\beta}t^{2}A$.

Since $A(u)>0$ and $B(u)>0$ for $u\ne0$, and $\beta>2$ by Lemma \ref{lem:sigma},
the equation $h_u'(t)=0$ reads $t^{\beta-2}=\frac{(\alpha+2)A}{\beta B}$ and has
the single positive root \eqref{eq:tu}. The sign of $h_u'(t)$ is the sign of
$\frac{(\alpha+2)A}{\beta B}-t^{\beta-2}$, which is positive for $t<t_u$ and
negative for $t>t_u$. This gives the monotonicity, the strict maximality of $t_u$ and the
statements about the sign of $\Qa(u^{t})$. The behaviour at $0$ and at infinity
is read off \eqref{eq:hu}, again because $\beta>2$. Taking $t=1$ in the sign
analysis, we obtain the last equivalences.

It remains to compute the maximum. By \eqref{eq:master} and $\Qa(u^{t_u})=0$, we
have $h_u(t_u)=\frac{\beta-2}{2\beta}t_u^{2}A$. Inserting \eqref{eq:tu}, we obtain
\begin{equation*}
h_u(t_u)=\frac{\beta-2}{2\beta}\Big(\frac{\alpha+2}{\beta}\Big)^{\frac{2}{\beta-2}}
A^{1+\frac{2}{\beta-2}}B^{-\frac{2}{\beta-2}}
=K_{\alpha,N,b}\,\frac{A^{\frac{\beta}{\beta-2}}}{B^{\frac{2}{\beta-2}}}.
\end{equation*}
We write $B=J(u)A^{\beta/2}\|u\|_{2}^{\sigma_c}$. Then the powers of $A$ cancel:
\begin{equation*}
\frac{A^{\frac{\beta}{\beta-2}}}{B^{\frac{2}{\beta-2}}}
=\frac{A^{\frac{\beta}{\beta-2}}}
{J(u)^{\frac{2}{\beta-2}}A^{\frac{\beta}{\beta-2}}\|u\|_{2}^{\frac{2\sigma_c}{\beta-2}}}
=J(u)^{-\frac{2}{\beta-2}}\|u\|_{2}^{-\gamma},
\end{equation*}
by the definition \eqref{eq:betasig} of $\gamma$. This is \eqref{eq:maxvalue}.
\end{proof}

The expression of $h_u(t_u)$ in terms of $A(u)$ and $B(u)$ is
\cite[Lemma~5.2]{HLW2024}, where $K_{\alpha,N,b}$ is denoted by $C_{\alpha,N,b}$.
The new point is the last expression in \eqref{eq:maxvalue}. The cancellation of
$A$ in the last display drives the whole section. On each
fibre $\{u^{t}:t>0\}$, the energy has a maximum whose value is a decreasing
function of the Gagliardo--Nirenberg quotient. The size of $u$ enters only through
its mass. Minimizing over the sphere therefore amounts to maximizing $J$.

\begin{lemma}[Minimax characterization, {\cite[Lemma~5.3]{HLW2024}}]\label{lem:minimax}
For every $c>0$,
\begin{equation}\label{eq:minimax}
m_a(c)=\inf_{u\in S(c)}\ \max_{t>0}\ \Ea(u^{t}).
\end{equation}
\end{lemma}

\begin{proof}
If $u\in V(c)$, then $t_u=1$ by Lemma \ref{lem:fibering}, so
$\Ea(u)=\max_{t>0}\Ea(u^{t})$, which is at least the right-hand side of
\eqref{eq:minimax}. Taking the infimum over $V(c)$, we obtain one inequality.
Conversely, for $u\in S(c)$, the function $u^{t_u}$ belongs to $V(c)$, since
dilations preserve the mass and $\Qa(u^{t_u})=0$. Hence
$\max_{t>0}\Ea(u^{t})=\Ea(u^{t_u})\ge m_a(c)$. Taking the infimum over $S(c)$, we
obtain the other inequality.
\end{proof}

\begin{proof}[Proof of Theorem \ref{thm:level}, parts (i) and (ii)]
By Lemma \ref{lem:minimax}, formula \eqref{eq:maxvalue} and Lemma
\ref{lem:GNrestrict}, we get
\begin{equation*}
m_a(c)=\inf_{u\in S(c)}K_{\alpha,N,b}J(u)^{-\frac{2}{\beta-2}}c^{-\gamma}
=K_{\alpha,N,b}\,c^{-\gamma}\Big(\sup_{u\in S(c)}J(u)\Big)^{-\frac{2}{\beta-2}}
=K_{\alpha,N,b}\,\GN^{-\frac{2}{\beta-2}}\,c^{-\gamma},
\end{equation*}
where we used that $s\mapsto s^{-\frac{2}{\beta-2}}$ is continuous and decreasing
on $(0,\infty)$. Since $\frac{2}{\beta-2}=\frac{4}{N\alpha+2b-4}$, this is
\eqref{eq:levelformula}. The properties of $c\mapsto m_a(c)$ listed in (i) are
those of the power function $c^{-\gamma}$ with $\gamma>0$.

For (ii), the attainment of $\GN$ is part of Proposition \ref{prop:GN}. Let
$u\in V(c)$. By Lemma \ref{lem:fibering}, we have $t_u=1$ and
$\Ea(u)=K_{\alpha,N,b}J(u)^{-\frac{2}{\beta-2}}c^{-\gamma}$. Comparing with
\eqref{eq:levelformula}, we infer that $\Ea(u)=m_a(c)$ if and only if $J(u)=\GN$.
If $W$ is an optimizer with $\|W\|_{2}=c$, then $W^{t}$ is an optimizer for every
$t>0$ by Lemma \ref{lem:GNrestrict}, and $W^{t_W}\in V(c)$. By what we just
proved, $W^{t_W}$ is a minimizer of $m_a(c)$. Optimizers of mass $c$ exist,
because any optimizer can be brought to mass $c$ by multiplication by a positive
constant (Lemma \ref{lem:GNrestrict}). Hence $m_a(c)$ is attained. Conversely, a
minimizer $u$ is an optimizer lying on $V(c)$, so $t_u=1$ and $u=u^{t_u}$ is of
the stated form.
\end{proof}

The value \eqref{eq:levelformula} itself does not require any compactness. It is
the infimum of an explicit decreasing function of $J$, and it holds whether or not
$\GN$ is attained. Compactness enters only through the existence of minimizers,
by way of Proposition \ref{prop:GN}.

We now turn to the ground-state property, in the sense of Definition
\ref{def:gs}. The first step is to observe that optimizers are solutions, up to a
normalization.

\begin{lemma}\label{lem:EL}
Let $W$ be an optimizer of \eqref{eq:GN}. Then $W$ is a weak solution of
\begin{equation*}
\Lc W+\frac{\sigma_c A(W)}{\beta\,\Mass(W)}\,W
=\frac{(\alpha+2)A(W)}{\beta\,B(W)}\,|x|^{-b}|W|^{\alpha}W .
\end{equation*}
In particular, if in addition $\Qa(W)=0$, then $W$ solves \eqref{eq:stationary}
with $\omega=\frac{\sigma_cA(W)}{\beta\Mass(W)}>0$. Conversely, if the optimizer
$W$ is a weak solution of \eqref{eq:stationary} for some $\omega\in\R$, then
$\Qa(W)=0$ and $\omega=\frac{\sigma_cA(W)}{\beta\Mass(W)}$.
\end{lemma}

\begin{proof}
The functionals $A$ and $\Mass$ are smooth on $H^1$, and $B$ is of class $C^1$ by
Lemma \ref{lem:weighted}(a). All three are positive at $W\ne0$. Hence $\log J$ is
differentiable at $W$ in the real sense. Since $W$ maximizes $J$ on the open set
$H^1\setminus\{0\}$, this derivative vanishes at $W$. Thus, for every $h\in H^1$,
we have
\begin{equation*}
\frac{B'(W)h}{B(W)}-\frac{\beta}{2}\,\frac{A'(W)h}{A(W)}-\frac{\sigma_c}{2}\,
\frac{\Mass'(W)h}{\Mass(W)}=0 .
\end{equation*}
With $A'(W)h=2\Ree\langle W,h\rangle_a$, $\Mass'(W)h=2\Ree\int W\bar h$ and
$B'(W)h=(\alpha+2)\Ree\int|x|^{-b}|W|^{\alpha}W\bar h$, multiplying by
$\frac{A(W)}{\beta}$, we get
\begin{equation*}
\Ree\Big(\langle W,h\rangle_a+\frac{\sigma_cA(W)}{\beta\Mass(W)}\int W\bar h\Big)
=\frac{(\alpha+2)A(W)}{\beta B(W)}\Ree\int|x|^{-b}|W|^{\alpha}W\bar h .
\end{equation*}
Replacing $h$ by $ih$, we obtain the same identity for the imaginary parts. This
is the weak form of the equation. If $\Qa(W)=0$, then
$A(W)=\frac{\beta}{\alpha+2}B(W)$ by \eqref{eq:EQAB}, and the coefficient on the
right-hand side equals $1$.

For the converse, assume that $W$ also solves \eqref{eq:stationary} for some
$\omega\in\R$, and put
\begin{equation*}
\mu:=\frac{\sigma_cA(W)}{\beta\Mass(W)}-\omega,\qquad
\nu:=\frac{(\alpha+2)A(W)}{\beta B(W)}-1 .
\end{equation*}
Subtracting the two weak formulations, we obtain
$\int\big(\mu W-\nu|x|^{-b}|W|^{\alpha}W\big)\bar h\dd=0$ for every
$h\in C_c^{\infty}(\R^N\setminus\{0\})$. Since the function in parentheses belongs
to $L^{1}_{loc}(\R^N\setminus\{0\})$, we infer that
$\mu W=\nu|x|^{-b}|W|^{\alpha}W$ almost everywhere. Suppose that $\nu\ne0$. The set
$E:=\{W\ne0\}$ has positive measure, because $W\ne0$, and on $E$ we have
$|x|^{-b}|W|^{\alpha}=\mu/\nu$. Hence $\mu/\nu>0$ and $|W|=\kappa|x|^{b/\alpha}$ on
$E$, with $\kappa:=(\mu/\nu)^{1/\alpha}$. Consequently,
$\mathbf 1_E=\kappa^{-1}|x|^{-b/\alpha}|W|$ almost everywhere on $\R^N\setminus\{0\}$.
Since $|W|\in H^1(\R^N)$ and $|x|^{-b/\alpha}$ is smooth on $\R^N\setminus\{0\}$, we
deduce that $\mathbf 1_E\in H^1_{loc}(\R^N\setminus\{0\})$. A function of class
$W^{1,1}_{loc}$ has a vanishing gradient almost everywhere on each of its level
sets. Therefore $\nabla\mathbf 1_E=0$ almost everywhere, and $\mathbf 1_E$ is
almost everywhere constant on the connected open set $\R^N\setminus\{0\}$. As $E$
has positive measure, it follows that $|W|=\kappa|x|^{b/\alpha}$ almost everywhere
on $\R^N$, which contradicts $W\in L^{2}(\R^N)$. Thus $\nu=0$, that is,
$A(W)=\frac{\beta}{\alpha+2}B(W)$, which means $\Qa(W)=0$ by \eqref{eq:EQAB}.
Finally, $\mu W=0$ gives $\mu=0$.
\end{proof}

Testing Definition \ref{def:weak} with $h=v$ shows that every weak solution of
\eqref{eq:stationary} satisfies the Nehari identity
\begin{equation}\label{eq:nehari}
A(v)+\omega\Mass(v)=B(v).
\end{equation}
The next proposition gives a sharp lower bound for the action on the whole Nehari
set. Its proof uses only \eqref{eq:GN} and an optimization in two variables. In
particular, it does not use the Pohozaev identity (see Remark \ref{rmk:pohozaev}).

\begin{proposition}\label{prop:leastaction}
Let $\omega>0$ and let $v\in\Ha(\R^{N})\setminus\{0\}$ satisfy \eqref{eq:nehari}.
This is the case, in particular, for every nontrivial solution of
\eqref{eq:stationary}. Then $\Sw(v)=\Theta_{\alpha}B(v)$ with
$\Theta_{\alpha}:=\frac{\alpha}{2(\alpha+2)}>0$, and
\begin{equation}\label{eq:Bbound}
B(v)\ \ge\ B_\omega:=
\Big[\GN\Big(\frac{\beta}{\alpha+2}\Big)^{\frac{\beta}{2}}
\Big(\frac{\sigma_c}{(\alpha+2)\omega}\Big)^{\frac{\sigma_c}{2}}\Big]^{-\frac{2}{\alpha}}.
\end{equation}
Equality holds in \eqref{eq:Bbound} if and only if $v$ is an optimizer of
\eqref{eq:GN} with $\Qa(v)=0$, and in that case
\begin{equation}\label{eq:ABmsol}
A(v)=\frac{\beta}{\alpha+2}B(v),\qquad
\Mass(v)=\frac{\sigma_c}{(\alpha+2)\,\omega}B(v),\qquad
\Sw(v)=\Theta_{\alpha}\,B(v).
\end{equation}
\end{proposition}

\begin{proof}
By \eqref{eq:nehari}, we have
$\Sw(v)=\frac12\big(A(v)+\omega\Mass(v)\big)-\frac{B(v)}{\alpha+2}
=\big(\frac12-\frac1{\alpha+2}\big)B(v)=\Theta_\alpha B(v)$. Put
$x:=A(v)/B(v)$ and $y:=\omega\Mass(v)/B(v)$, so that $x,y>0$ and $x+y=1$. Then
$A(v)^{\beta/2}\Mass(v)^{\sigma_c/2}=\omega^{-\sigma_c/2}x^{\beta/2}y^{\sigma_c/2}
B(v)^{\frac{\beta+\sigma_c}{2}}$, and $\beta+\sigma_c=\alpha+2$. Inserting this in
\eqref{eq:GN} and dividing by $B(v)^{\frac{\alpha+2}{2}}$, we get
\begin{equation*}
B(v)^{-\frac{\alpha}{2}}\le\GN\,\omega^{-\frac{\sigma_c}{2}}\,x^{\frac\beta2}y^{\frac{\sigma_c}{2}}.
\end{equation*}
The function $x\mapsto\frac\beta2\log x+\frac{\sigma_c}{2}\log(1-x)$ is strictly
concave on $(0,1)$ and attains its maximum only at $x=\frac{\beta}{\beta+\sigma_c}
=\frac{\beta}{\alpha+2}$. Therefore
$x^{\beta/2}y^{\sigma_c/2}\le\big(\frac{\beta}{\alpha+2}\big)^{\beta/2}
\big(\frac{\sigma_c}{\alpha+2}\big)^{\sigma_c/2}$, and raising the resulting
inequality to the power $-\frac2\alpha$, which reverses it, we obtain
\eqref{eq:Bbound}. Equality holds exactly when \eqref{eq:GN} is an equality for $v$
and $x=\frac{\beta}{\alpha+2}$; the latter means $A(v)=\frac{\beta}{\alpha+2}B(v)$,
that is, $\Qa(v)=0$. In that case $\omega\Mass(v)=B(v)-A(v)=\frac{\sigma_c}{\alpha+2}B(v)$,
which completes \eqref{eq:ABmsol}.
\end{proof}

\begin{remark}\rm \label{rmk:pohozaev}
For solutions of \eqref{eq:ellipticQ}, the Pohozaev identity $\Qa(v)=0$ is
contained in \cite[Lemma~3.3]{CamposGuzman2021}, where it is combined with the
Nehari identity \eqref{eq:nehari}. It is also stated, for every $N\ge3$, in
\cite[Lemma~5.1]{HLW2024}. Formally, it is obtained by pairing the equation with
$x\cdot\nabla v$. A rigorous derivation requires some regularity of $v$ away from
the origin and the control of boundary terms near $0$ and at infinity. Combined with \eqref{eq:nehari}, it shows that no
nontrivial solution exists for $\omega\le0$ \cite[Lemma~5.8]{HLW2024}. None of our proofs uses this
identity. For ground states, $\Qa=0$ follows directly from Theorem
\ref{thm:level}(ii)--(iii).
\end{remark}

\begin{proof}[Proof of Theorem \ref{thm:level}, part (iii)]
Let $u_c$ be a minimizer of $m_a(c)$. By part (ii) it is an optimizer with
$\Qa(u_c)=0$. Lemma \ref{lem:EL} shows that it solves \eqref{eq:stationary}
with $\omega=\omega_c:=\frac{\sigma_cA(u_c)}{\beta c^{2}}>0$. We compute this
number. From equality in \eqref{eq:GN}, $A(u_c)=\frac{\beta}{\alpha+2}B(u_c)$ and
$\Mass(u_c)=c^{2}$, we get
\begin{equation*}
B(u_c)=\GN\Big(\frac{\beta}{\alpha+2}\Big)^{\frac\beta2}B(u_c)^{\frac\beta2}c^{\sigma_c},
\qquad\text{hence}\qquad
B(u_c)=\Big(\frac{\alpha+2}{\beta}\Big)^{\frac{\beta}{\beta-2}}
\GN^{-\frac{2}{\beta-2}}c^{-\gamma},
\end{equation*}
where we used $1-\frac\beta2=-\frac{\beta-2}{2}$ and
$\gamma=\frac{2\sigma_c}{\beta-2}$. Since
$\omega_c=\frac{\sigma_c}{(\alpha+2)c^{2}}B(u_c)$ by \eqref{eq:ABmsol}, we obtain
\begin{equation*}
\omega_c=\frac{\sigma_c}{\alpha+2}\Big(\frac{\alpha+2}{\beta}\Big)^{\frac{\beta}{\beta-2}}
\GN^{-\frac{2}{\beta-2}}c^{-\gamma-2},
\end{equation*}
which is \eqref{eq:omegac}, because $\gamma+2=\frac{4\alpha}{N\alpha+2b-4}$ by Lemma
\ref{lem:sigma}. By the converse part of Lemma \ref{lem:EL}, any optimizer $v$
which solves \eqref{eq:stationary} for some $\omega$ satisfies $\Qa(v)=0$, and the
same computation applies to it. By \eqref{eq:nehari}, its frequency is
necessarily $\omega_{\|v\|_2}$. The map
$c\mapsto\omega_c$ is a negative power of $c$, hence a decreasing bijection of
$(0,\infty)$ onto itself.

By Proposition \ref{prop:leastaction}, $u_c$ achieves equality in
\eqref{eq:Bbound} with $\omega=\omega_c$. Hence $\Sw(u_c)=\Theta_\alpha B_{\omega_c}\le
\Sw(v)$ for every nontrivial solution $v$ at the frequency $\omega_c$. Thus $u_c$
is a ground state, and the least action at frequency $\omega_c$ equals
$\Theta_\alpha B_{\omega_c}$.

Conversely, let $v$ be a ground state at some frequency $\omega>0$. Choose $c'>0$
with $\omega_{c'}=\omega$ and a minimizer $u_{c'}$ of $m_a(c')$. By the previous
paragraph, we have $\Sw(u_{c'})=\Theta_\alpha B_\omega$. It follows that the
least action at frequency $\omega$ is $\Theta_\alpha B_\omega$, and
$\Sw(v)=\Theta_\alpha B_\omega$.
Thus $v$ achieves equality in \eqref{eq:Bbound}. Proposition
\ref{prop:leastaction} shows that $v$ is an optimizer with $\Qa(v)=0$, that is,
an optimizer in $V(\|v\|_2)$. By part (ii), it is a minimizer of $m_a(\|v\|_2)$,
and by the remark made above, $\omega=\omega_{\|v\|_2}$.
\end{proof}

\begin{remark}\rm 
Let $\varphi$ be a ground state of \eqref{eq:stationary} with $\omega=1$. By
Theorem \ref{thm:level}(iii) and Proposition \ref{prop:leastaction}, $\varphi$ is
an optimizer of \eqref{eq:GN} and satisfies \eqref{eq:ABmsol} with $\omega=1$,
that is, $A(\varphi)=\frac{\beta}{\alpha+2}B(\varphi)$ and
$\Mass(\varphi)=\frac{\sigma_c}{\alpha+2}B(\varphi)$. These relations coincide
with the Pohozaev identities of \cite[(3.3)--(3.4)]{CamposGuzman2021}, written
with the Nehari identity. Since $\beta+\sigma_c=\alpha+2$, equality in
\eqref{eq:GN} gives
\begin{equation*}
\GN=\frac{B(\varphi)}{A(\varphi)^{\beta/2}\Mass(\varphi)^{\sigma_c/2}}
=\frac{\alpha+2}{\sigma_c}\Big(\frac{\sigma_c}{\beta}\Big)^{\frac{\beta}{2}}
\Mass(\varphi)^{1-\frac{\beta+\sigma_c}{2}}
=\frac{\alpha+2}{\sigma_c}\Big(\frac{\sigma_c}{\beta}\Big)^{\frac{\beta}{2}}
\|\varphi\|_{2}^{-\alpha}.
\end{equation*}
This is the analogue of Weinstein's formula \cite{Weinstein1983}. It agrees with
\eqref{eq:omegac} for $\omega_c=1$ and with \cite[(2.7)]{BGS2023}. It need not hold for solutions of
\eqref{eq:stationary} that are not ground states, since these need not be
optimizers.
\end{remark}

In particular, ground states exist at every frequency $\omega>0$, and the
number $\Mass(Q)^{\sigma}\Ea(Q)$ does not depend on which ground state $Q$ is
chosen. We now identify the normalized level with the classical threshold.

\begin{proof}[Proof of Theorem \ref{thm:equiv}]
Let $Q$ be a ground state of \eqref{eq:ellipticQ}, that is, of
\eqref{eq:stationary} with $\omega=1$, and put $c_Q:=\|Q\|_{2}$. By Theorem
\ref{thm:level}(iii), $Q$ is a minimizer of $m_a(c_Q)$, so $\Ea(Q)=m_a(c_Q)$.
Since $c\mapsto c^{\gamma}m_a(c)$ is constant, we have
\begin{equation}\label{eq:mcQ}
m_a(c)=\Ea(Q)\Big(\frac{\|Q\|_{2}}{c}\Big)^{\gamma},\qquad c>0 .
\end{equation}
Hence, for $u_0$ of mass $c$, the inequality $\Ea(u_0)<m_a(c)$ is equivalent to
$c^{\gamma}\Ea(u_0)<\|Q\|_{2}^{\gamma}\Ea(Q)$, that is, since $c^\gamma=\Mass(u_0)^{\gamma/2}$,
to $\Mass(u_0)^{\sigma}\Ea(u_0)<\Mass(Q)^{\sigma}\Ea(Q)$. This is
\eqref{eq:equivE}, and $\sigma=\frac{1-s_c}{s_c}$ by Lemma \ref{lem:sigma}. When
$\Ea(u_0)>0$, both sides of \eqref{eq:equivE} are positive (note that
$\Ea(Q)=m_a(c_Q)>0$). Raising them to the power $s_c\in(0,1)$ preserves the
inequality. Since $\sigma s_c=1-s_c$, we obtain the form
$\Mass(u_0)^{1-s_c}\Ea(u_0)^{s_c}<\Mass(Q)^{1-s_c}\Ea(Q)^{s_c}$. Raising to the
power $1/s_c$, we go back.

For \eqref{eq:equivQ}, we introduce
\begin{equation}\label{eq:Astar}
A_*(c):=\frac{2\beta}{\beta-2}m_a(c)
=\Big(\frac{\alpha+2}{\beta\,\GN\,c^{\sigma_c}}\Big)^{\frac{2}{\beta-2}}
=\|Q\|_{\dHa}^{2}\Big(\frac{\|Q\|_{2}}{c}\Big)^{\gamma}.
\end{equation}
The second expression follows from \eqref{eq:levelformula} and \eqref{eq:K}, and
the third from \eqref{eq:mcQ} and $\Ea(Q)=\frac{\beta-2}{2\beta}A(Q)$, which is
\eqref{eq:master} at $t=1$ for $Q\in V(c_Q)$. We first show that
\begin{equation}\label{eq:Qneg}
\Qa(u_0)\le0\quad\Longrightarrow\quad A(u_0)\ge A_*(c).
\end{equation}
Indeed, $\Qa(u_0)\le0$ means $A(u_0)\le\frac{\beta}{\alpha+2}B(u_0)$, and by
\eqref{eq:GN} this gives
$A(u_0)\le\frac{\beta\GN}{\alpha+2}A(u_0)^{\frac\beta2}c^{\sigma_c}$, that is,
$A(u_0)^{1-\frac\beta2}\le\frac{\beta\GN}{\alpha+2}c^{\sigma_c}$. Since
$1-\frac\beta2<0$, raising to the power $-\frac{2}{\beta-2}$ reverses the
inequality, and we get $A(u_0)\ge A_*(c)$.

Assume now that $\Ea(u_0)<m_a(c)$. If $\Qa(u_0)>0$, then \eqref{eq:master} at
$t=1$ gives $\frac{\beta-2}{2\beta}A(u_0)<\Ea(u_0)<m_a(c)$, whence $A(u_0)<A_*(c)$.
Conversely, if $A(u_0)<A_*(c)$, then $\Qa(u_0)>0$ by \eqref{eq:Qneg}. So, below the
level, $\Qa(u_0)>0$ is equivalent to $A(u_0)<A_*(c)$, which by the last expression
in \eqref{eq:Astar} reads $c^{\gamma}A(u_0)<\|Q\|_{2}^{\gamma}\|Q\|_{\dHa}^{2}$,
that is, $\|u_0\|_{2}^{\sigma}\|u_0\|_{\dHa}<\|Q\|_{2}^{\sigma}\|Q\|_{\dHa}$. This is
\eqref{eq:equivQ}.
\end{proof}

\begin{proof}[Proof of Proposition \ref{prop:monotone}]
Let $\aH<a_1<a_2$, and denote by $A_j$, $J_j$ and $\mathcal{C}_{a_j}$ the
quantities $A$, $J$ and $\GN$ computed with $a=a_j$. For $u\in\Ha\setminus\{0\}$,
we have
\begin{equation*}
A_1(u)=A_2(u)-(a_2-a_1)\int\frac{|u|^{2}}{|x|^{2}}\dd<A_2(u),
\end{equation*}
the integral being finite by \eqref{eq:hardy} and positive. Since $\beta>0$, we
infer that $J_1(u)>J_2(u)$. By Proposition \ref{prop:GN}, the constant
$\mathcal{C}_{a_2}$ is attained at some $W$, and then
$\mathcal{C}_{a_1}\ge J_1(W)>J_2(W)=\mathcal{C}_{a_2}$. Since $\frac{2}{\beta-2}>0$
and the constants $K_{\alpha,N,b}$ and $\Lambda_{\alpha,N,b}$ in
\eqref{eq:levelformula} and \eqref{eq:omegac} do not depend on $a$, both
$m_a(c)$ and $\omega_c$ are strictly increasing functions of $a$.
\end{proof}

\begin{remark}\rm \label{rmk:monotoneset}
Proposition \ref{prop:monotone} concerns the threshold value. It does not by
itself imply that the set of data satisfying \eqref{eq:below} and $\Qa(u_0)>0$
grows with $a$, because the energy $\Ea(u_0)$ increases with $a$ as well. We do
not know whether this set is monotone in $a$.
\end{remark}

%%%%%%%%%%%%%%%%%%%%%%%%%%%%%%%%%%%%%%%%%%%%%%%%%%%%%%%%%%%%%%%%%%%%%%%%%%%%%%
\section{The dynamical dichotomy below the normalized level}\label{sec:dynamics}
%%%%%%%%%%%%%%%%%%%%%%%%%%%%%%%%%%%%%%%%%%%%%%%%%%%%%%%%%%%%%%%%%%%%%%%%%%%%%%

The passage from the variational analysis to the dynamics goes through the
following coercivity statement, which is once more a consequence of
\eqref{eq:master}. It holds pointwise in time, with constants that are explicit in
the energy gap. Hence no continuity argument is needed to derive it.

\begin{lemma}\label{lem:coercivity}
Let $N\ge3$, $a>\aH$, $0<b<2$, assume \eqref{eq:range}, and let $c>0$ and
$u\in S(c)$.
\begin{enumerate}
\item[(i)] If $\Qa(u)<0$, then
\begin{equation}\label{eq:coerneg}
\Qa(u)<\beta\big(\Ea(u)-m_a(c)\big).
\end{equation}
\item[(ii)] If $\Qa(u)>0$, then
\begin{equation}\label{eq:coerpos}
0<\frac{\beta-2}{2\beta}\,\|u\|_{\dHa}^{2}<\Ea(u).
\end{equation}
\item[(iii)] If $\Ea(u)<m_a(c)$, then $\Qa(u)\ne0$.
\item[(iv)] If $\Qa(u)>0$ and $\Ea(u)\le(1-\eta)m_a(c)$ for some $\eta\in(0,1)$,
then, with $A_*(c)$ as in \eqref{eq:Astar},
\begin{equation}\label{eq:unifcoer}
\|u\|_{\dHa}^{2}<(1-\eta)A_*(c)
\qquad\text{and}\qquad
\Qa(u)\ge\Big(1-(1-\eta)^{\frac{\beta-2}{2}}\Big)\|u\|_{\dHa}^{2}.
\end{equation}
\end{enumerate}
\end{lemma}

\begin{proof}
Assertion (iii) is immediate: $\Qa(u)=0$ places $u$ in $V(c)$, and we get
$\Ea(u)\ge m_a(c)$ by \eqref{eq:mc}. Assertion (ii) is \eqref{eq:master} at $t=1$,
namely $\Ea(u)=\frac{\beta-2}{2\beta}A(u)+\frac1\beta\Qa(u)$, together with
$\Qa(u)>0$ and $A(u)>0$.

For (i), let $t_u$ be the maximum point of the fibering map given by Lemma
\ref{lem:fibering}. Since $\Qa(u)<0$, we have $t_u<1$. Since $u^{t_u}\in V(c)$,
we have $\Ea(u^{t_u})\ge m_a(c)$. On the other hand, \eqref{eq:master} at $t=t_u$
gives $\Ea(u^{t_u})=\frac{\beta-2}{2\beta}t_u^{2}A(u)$. Combining the two and using
$t_u<1$ and $m_a(c)>0$, we obtain
\begin{equation*}
\frac{\beta-2}{2\beta}A(u)=\frac{\Ea(u^{t_u})}{t_u^{2}}\ \ge\ \frac{m_a(c)}{t_u^{2}}\ >\ m_a(c).
\end{equation*}
Writing \eqref{eq:master} at $t=1$ as
$\Qa(u)=\beta\big(\Ea(u)-\frac{\beta-2}{2\beta}A(u)\big)$ we obtain
$\Qa(u)<\beta\big(\Ea(u)-m_a(c)\big)$.

For (iv), assertion (ii) gives $\frac{\beta-2}{2\beta}A(u)<\Ea(u)\le(1-\eta)m_a(c)$.
This is the first inequality in \eqref{eq:unifcoer}, by the definition of
$A_*(c)$. For the second one, we use \eqref{eq:GN} and the second expression in
\eqref{eq:Astar}, which says that
$\frac{\beta\GN}{\alpha+2}c^{\sigma_c}=A_*(c)^{-\frac{\beta-2}{2}}$. We get
\begin{equation*}
\Qa(u)=A(u)-\frac{\beta}{\alpha+2}B(u)\ge A(u)-\frac{\beta\GN}{\alpha+2}c^{\sigma_c}A(u)^{\frac\beta2}
=A(u)\Big(1-\Big(\frac{A(u)}{A_*(c)}\Big)^{\frac{\beta-2}{2}}\Big),
\end{equation*}
and the bracket is at least $1-(1-\eta)^{\frac{\beta-2}{2}}$.
\end{proof}

\begin{lemma}\label{lem:invariance}
Let $N\ge3$, $a>\aH$, $0<b<2$, assume \eqref{eq:range} and the first part of
Assumption \ref{ass:cauchy}. Let
$u_0\in\Ha\setminus\{0\}$ satisfy \eqref{eq:below}, and let $u$ be the maximal
solution. Then $\Qa(u(t))\ne0$ for every $t\in(-T_-,T_+)$, and the sign of
$\Qa(u(t))$ is constant. If $\Qa(u_0)>0$, the solution is global,
$0<\Ea(u_0)<m_a(c)$, and
\begin{equation}\label{eq:unifbound}
\sup_{t\in\R}\|u(t)\|_{\dHa}^{2}\le\frac{2\beta}{\beta-2}\Ea(u_0)<A_*(c),
\qquad
\inf_{t\in\R}\frac{\Qa(u(t))}{\|u(t)\|_{\dHa}^{2}}\ge1-\Big(\frac{\Ea(u_0)}{m_a(c)}\Big)^{\frac{\beta-2}{2}}>0 .
\end{equation}
If $\Qa(u_0)<0$, then
\begin{equation}\label{eq:unifneg}
\Qa(u(t))<\beta\big(\Ea(u_0)-m_a(c)\big)<0
\qquad\text{for all }t\in(-T_-,T_+).
\end{equation}
\end{lemma}

\begin{proof}
By conservation of mass and energy, we have $u(t)\in S(c)$ and
$\Ea(u(t))=\Ea(u_0)<m_a(c)$ for every $t$ in the maximal interval. Lemma
\ref{lem:coercivity}(iii) gives $\Qa(u(t))\ne0$. The map $t\mapsto\Qa(u(t))$ is
continuous, because $u\in C((-T_-,T_+);\Ha)$ and $A$ and $B$ are continuous on
$\Ha$ (Lemma \ref{lem:weighted}). A continuous real function without zeros on an
interval has constant sign.

If $\Qa(u_0)>0$, then $\Qa(u(t))>0$ for all $t$, and Lemma \ref{lem:coercivity}(ii)
together with the conservation of energy gives $\Ea(u_0)>0$ and
$\frac{\beta-2}{2\beta}\|u(t)\|^{2}_{\dHa}<\Ea(u_0)$. Thus $\|u(t)\|_{\dHa}$
stays bounded, and $T_\pm=\infty$ by the blow-up alternative. Applying Lemma
\ref{lem:coercivity}(iv) with $1-\eta=\Ea(u_0)/m_a(c)$, we obtain
\eqref{eq:unifbound}. If $\Qa(u_0)<0$, then $\Qa(u(t))<0$ for all $t$. Applying
\eqref{eq:coerneg} to $u(t)\in S(c)$, we get \eqref{eq:unifneg}.
\end{proof}

For the scattering half, we use the scattering theorem of Hajaiej, Luo and Wang.
In view of Remark \ref{rmk:exponent}, we state it in scale-invariant form. This
is the one point of the paper where we rely on a result whose hypotheses, as
published, differ from the form in which we use it.

\begin{proposition}[Scattering below the ground state, {\cite[Theorem~1.1]{HLW2024}}]\label{prop:scatcrit}
Assume \eqref{eq:dynrange}, and let $Q$ be a ground state of
\eqref{eq:ellipticQ}. If $u_0\in\Ha(\R^3)\setminus\{0\}$ satisfies
\begin{equation*}
\Mass(u_0)^{\sigma}\Ea(u_0)<\Mass(Q)^{\sigma}\Ea(Q)
\qquad\text{and}\qquad
\|u_0\|_{2}^{\sigma}\|u_0\|_{\dHa}<\|Q\|_{2}^{\sigma}\|Q\|_{\dHa},
\end{equation*}
then the solution of \eqref{eq:main} is global and scatters in $\Ha(\R^3)$.
\end{proposition}

Both conditions are invariant under \eqref{eq:scaling}, and no sign condition on
the energy is needed. By the last expression in \eqref{eq:Astar}, the second
condition says that $\|u_0\|_{\dHa}^{2}<A_*(c)$ with $c=\|u_0\|_2$. By
\eqref{eq:Qneg}, this gives $\Qa(u_0)>0$, and then $\Ea(u_0)>0$ by
\eqref{eq:coerpos}. When $\Ea(u_0)>0$, the first condition is equivalent to
$\Mass(u_0)^{1-s_c}\Ea(u_0)^{s_c}<\Mass(Q)^{1-s_c}\Ea(Q)^{s_c}$. This is the form
used in \cite{FarahGuzman2017,MiaoMurphyZheng2021,CamposGuzman2021}. In
\cite{HLW2024}, the threshold is written with a positive optimizer $Q_{a,b}$ of
\eqref{eq:GN} which solves \eqref{eq:ellipticQ}. By the converse part of Lemma
\ref{lem:EL}, we have $\Qa(Q_{a,b})=0$. Hence $Q_{a,b}$ is a minimizer of
$m_a(\|Q_{a,b}\|_2)$ by Theorem \ref{thm:level}(ii), and a ground state by Theorem
\ref{thm:level}(iii). Since the two quantities in Proposition \ref{prop:scatcrit}
take the same value for all ground states, the choice of $Q$ is immaterial.

The proof in \cite{HLW2024} follows the virial--Morawetz scheme of
\cite{DodsonMurphy2017,CamposCardoso2022}. The global existence of the solution
is taken from \cite{CamposGuzman2021}, and the scattering criterion
\cite[Lemma~1.3]{HLW2024} only uses a uniform bound in $\Ha$. The threshold
hypotheses enter the proof only through \cite[Proposition~2.5]{HLW2024}, that
is, through uniform bounds along the flow: a bound on $\|u(t)\|_{\dHa}$ strictly
below the ground-state value, and a coercivity estimate of the type
\eqref{eq:unifbound}, applied also to spatial truncations of $u(t)$ in
\cite[(4.9)]{HLW2024}. Written in terms of $\Mass^{\sigma}\Ea$, these bounds are invariant under
\eqref{eq:scaling}. Below the level, Lemma \ref{lem:invariance} provides the first
two bounds directly in terms of $m_a(c)$. The coercivity of the truncations is
proved below; see \eqref{eq:QaMB} and the estimate that follows it.

The formulation of \cite{BGS2023} through the potential energy is well suited to
truncations. Let $Q$ be a ground state of \eqref{eq:ellipticQ};
by Theorem \ref{thm:level} it is an optimizer of \eqref{eq:GN} with $\Qa(Q)=0$.
For $v\in\Ha\setminus\{0\}$, we write $B(v)=B(v)^{2/\beta}B(v)^{(\beta-2)/\beta}$,
bound the first factor by \eqref{eq:GN}, and use $\sigma(\beta-2)=\sigma_c$, which
is \eqref{eq:sigmaid}. We obtain
$B(v)\le\GN^{2/\beta}\big(\Mass(v)^{\sigma}B(v)\big)^{\frac{\beta-2}{\beta}}A(v)$.
Every step is an equality for $v=Q$, and $B(Q)/A(Q)=\frac{\alpha+2}{\beta}$
because $\Qa(Q)=0$. Eliminating $\GN$, we obtain
\begin{equation}\label{eq:QaMB}
\frac{\beta}{\alpha+2}\,B(v)\le
\Big(\frac{\Mass(v)^{\sigma}B(v)}{\Mass(Q)^{\sigma}B(Q)}\Big)^{\frac{\beta-2}{\beta}}A(v),
\qquad\text{hence}\qquad
\Qa(v)\ge\Big(1-\Big(\frac{\Mass(v)^{\sigma}B(v)}{\Mass(Q)^{\sigma}B(Q)}\Big)^{\frac{\beta-2}{\beta}}\Big)A(v).
\end{equation}
Now let $u$ be as in Lemma \ref{lem:invariance} with $\Qa(u_0)>0$, and put
$c=\|u_0\|_{2}$. By \eqref{eq:unifbound} and the first expression of $A_*(c)$ in
\eqref{eq:Astar}, we have $A(u(t))\le\frac{\Ea(u_0)}{m_a(c)}A_*(c)$ for all $t$.
Inserting this in \eqref{eq:GN}, and using the last expression of $A_*(c)$ in
\eqref{eq:Astar}, the identity $\frac{\gamma\beta}{2}=\gamma+\sigma_c$ and equality
in \eqref{eq:GN} for $Q$, we get
\begin{equation*}
\Mass(u(t))^{\sigma}B(u(t))\le\GN\Big(\frac{\Ea(u_0)}{m_a(c)}\Big)^{\frac\beta2}
c^{\gamma+\sigma_c}A_*(c)^{\frac\beta2}
=\Big(\frac{\Ea(u_0)}{m_a(c)}\Big)^{\frac\beta2}\Mass(Q)^{\sigma}B(Q),
\qquad t\in\R .
\end{equation*}
The quantity $\Mass(v)^{\sigma}B(v)$ does not increase when $v$ is multiplied by
a smooth real cut-off $0\le\chi\le1$. Therefore, by \eqref{eq:QaMB}, we have, for
every such $\chi$ and every $t\in\R$,
\begin{equation*}
\Qa\big(\chi u(t)\big)\ge\Big(1-\Big(\frac{\Ea(u_0)}{m_a(c)}\Big)^{\frac{\beta-2}{2}}\Big)
\big\|\chi u(t)\big\|_{\dHa}^{2},
\end{equation*}
with the same constant as in \eqref{eq:unifbound}. This is the coercivity of
truncations required above. It replaces \cite[Proposition~2.5(ii)]{HLW2024}, which
is written there with exponent $1$ and in the form discussed in Remark
\ref{rmk:virialtypo}. The same bound shows that the hypothesis
$\sup_t\Mass(u(t))^{\sigma}B(u(t))<\Mass(\varphi)^{\sigma}B(\varphi)$ of
\cite[Theorem~2.2]{BGS2023} holds below $m_a(c)$ when $\Qa(u_0)>0$.

\begin{proof}[Proof of Theorem \ref{thm:dichotomy}]
By Assumption \ref{ass:cauchy}, Lemma \ref{lem:invariance} applies. Hence
$\Qa(u_0)\ne0$, and the sign of $\Qa(u(t))$ is conserved. If $\Qa(u_0)>0$, the
solution is global by Lemma \ref{lem:invariance}.

\emph{The scattering regime.} Assume $\Qa(u_0)>0$ and \eqref{eq:dynrange}. By
Theorem \ref{thm:equiv}, hypothesis \eqref{eq:below} and $\Qa(u_0)>0$ are exactly
the two conditions of Proposition \ref{prop:scatcrit}. Hence the solution
scatters.

\emph{The blow-up regime.} Assume $\Qa(u_0)<0$ and $u_0\in\Sigma$. By Assumption
\ref{ass:cauchy} the solution stays in $\Sigma$, the variance
$V(t)=\int|x|^{2}|u(t,x)|^{2}\dd$ is of class $C^{2}$, and $V''(t)=8\Qa(u(t))$ by
Lemma \ref{lem:virial}. Inserting \eqref{eq:unifneg}, we get \eqref{eq:blowup}:
\begin{equation*}
V''(t)\le8\beta\big(\Ea(u_0)-m_a(c)\big)=:-2\delta<0,
\qquad t\in(-T_-,T_+).
\end{equation*}
Integrating twice, we obtain $V(t)\le V(0)+V'(0)\,t-\delta t^{2}$ on the maximal
interval.
The right-hand side becomes negative for $|t|$ large, while $V(t)\ge0$. Hence
$T_+<\infty$ and $T_-<\infty$, and by the blow-up alternative
$\|u(t)\|_{\dHa}\to\infty$ as $t\uparrow T_+$ and as $t\downarrow-T_-$.

\emph{Sharpness.} Let $u_c$ be a minimizer of $m_a(c)$, which exists by Theorem
\ref{thm:level}(ii). Then $u_c\in V(c)$, $\Ea(u_c)=m_a(c)$, $u_c\in\Sigma$ by
Proposition \ref{prop:decay} below, and $e^{i\omega_ct}u_c$ solves
\eqref{eq:main} by Theorem \ref{thm:level}(iii). It belongs to $C(\R;\Ha)\cap C^1(\R;H^{-1})$. By the
uniqueness contained in Assumption \ref{ass:cauchy}, it is the solution with data
$u_c$. This solution is
global but does not scatter. Indeed, put $r:=\frac{2N}{N-1}$, so that $(4,r)$ is
admissible and $H^{1}(\R^N)\subset L^{r}(\R^N)$. Suppose that
$\|u(t)-e^{-it\Lc}u_+\|_{\Ha}\to0$ as $t\to\infty$ for some $u_+\in\Ha$. Then
there is $T>0$ with
$\|e^{-it\Lc}u_+\|_{r}\ge\|u_c\|_{r}-\|u(t)-e^{-it\Lc}u_+\|_{r}\ge\frac12\|u_c\|_{r}>0$
for $t\ge T$, so that
$\|e^{-it\Lc}u_+\|_{L^4_tL^r_x([T,\infty)\times\R^N)}=\infty$. This contradicts
Proposition \ref{prop:strichartz} with $F=0$. The same argument applies as
$t\to-\infty$. Therefore the strict inequality in \eqref{eq:below} cannot be
replaced by a non-strict one.
\end{proof}

\begin{remark}\rm \label{rmk:comparison}
Three features of this argument are worth pointing out. First, the dichotomy is
exhaustive without further hypotheses, because the case $\Qa(u_0)=0$ is excluded by
the definition of the level itself. Second, the blow-up part uses no information
on the ground state beyond the number $m_a(c)$; in particular it does not use the
uniqueness of $Q$, which to our knowledge is open when $a\ne0$ and $b\ne0$. Third,
\eqref{eq:unifneg} is uniform in time with an explicit constant. Solving
$V(0)+V'(0)t-\delta t^{2}=0$, we obtain the bound
\begin{equation*}
\max\{T_+,T_-\}\ \le\
\frac{\big|V'(0)\big|+\sqrt{V'(0)^{2}+4\delta V(0)}}{2\delta},
\qquad \delta=4\beta\big(m_a(c)-\Ea(u_0)\big),
\end{equation*}
on the blow-up times. This upper bound deteriorates as the energy approaches the
level, like $\delta^{-1}$ in general and like $\delta^{-1/2}$ when $V'(0)=0$. It
says nothing about the actual blow-up time of solutions close to the level.
\end{remark}

%%%%%%%%%%%%%%%%%%%%%%%%%%%%%%%%%%%%%%%%%%%%%%%%%%%%%%%%%%%%%%%%%%%%%%%%%%%%%%
\section{Decay of solutions of the stationary problem and strong instability}\label{sec:instability}
%%%%%%%%%%%%%%%%%%%%%%%%%%%%%%%%%%%%%%%%%%%%%%%%%%%%%%%%%%%%%%%%%%%%%%%%%%%%%%

To apply the blow-up part of Theorem \ref{thm:dichotomy} to data close to a
normalized ground state, we need to know that such a state has finite variance.
We prove more, namely exponential decay in an integrated sense. The proof is an
Agmon-type argument performed directly on the weak formulation. It avoids any
discussion of the regularity of solutions near the singularity.

\begin{proposition}\label{prop:decay}
Let $N\ge3$, $a>\aH$, $0<b<2$ and $0<\alpha<\frac{4-2b}{N-2}$, and let
$v\in\Ha(\R^{N})$ be a weak solution of \eqref{eq:stationary} with $\omega>0$. Put
$\kappa:=\min\{\lambda_a,\omega\}$, with $\lambda_a$ as in \eqref{eq:equivnorm}.
There is $R\ge1$, depending on $v$, such that for every $\eps\in(0,\frac12\sqrt\kappa]$,
\begin{equation}\label{eq:decay}
\int_{\R^N}e^{2\eps|x|}\big(|\nabla v|^{2}+|v|^{2}\big)\dd
\ \le\ C_\eps\,\frac{e^{2\eps R}}{\kappa}\,B(v),
\end{equation}
with $C_\eps:=4(1+\eps^{2})$. In particular $v\in\Sigma$.
\end{proposition}

\begin{proof}
Fix $\eps\in(0,\frac12\sqrt\kappa]$, and for $n\in\N$ let
$\theta_n(x):=e^{\eps\min\{|x|,n\}}$. This is a bounded Lipschitz function with
$|\nabla\theta_n|\le\eps\theta_n$ almost everywhere, so $\theta_n^2v\in H^1$ is an
admissible test function in Definition \ref{def:weak}. We use the pointwise identity
$\Ree\big(\nabla v\cdot\nabla(\theta_n^{2}\bar v)\big)=|\nabla(\theta_nv)|^{2}-|\nabla\theta_n|^{2}|v|^{2}$,
valid for real $\theta_n$. With $w_n:=\theta_nv$, the real part of the weak
formulation becomes
\begin{equation}\label{eq:katoineq}
A(w_n)+\omega\|w_n\|_{2}^{2}
=\int|\nabla\theta_n|^{2}|v|^{2}\dd+\int\frac{|v|^{\alpha}|w_n|^{2}}{|x|^{b}}\dd .
\end{equation}
By \eqref{eq:equivnorm} the left-hand side is at least $\kappa\|w_n\|_{H^1}^{2}$.
The first term on the right is at most $\eps^{2}\|w_n\|_{2}^{2}\le\frac\kappa4\|w_n\|_{H^1}^2$.
For the second one, we split the integral at a radius $R\ge1$. On $\{|x|<R\}$, we
use $\theta_n\le e^{\eps R}$ to bound it by $e^{2\eps R}B(v)$. On $\{|x|\ge R\}$, we
use
$|x|^{-b}\le1$, H\"older's inequality with exponents $\frac{\alpha+2}{\alpha}$ and
$\frac{\alpha+2}{2}$, and the Sobolev embedding $H^1\subset L^{\alpha+2}$, which
holds since $2<\alpha+2<\frac{2N}{N-2}$. We get
\begin{equation*}
\int_{|x|\ge R}\frac{|v|^{\alpha}|w_n|^{2}}{|x|^{b}}\dd
\le\|v\|_{L^{\alpha+2}(|x|\ge R)}^{\alpha}\|w_n\|_{\alpha+2}^{2}
\le C_S^2\,\|v\|_{L^{\alpha+2}(|x|\ge R)}^{\alpha}\|w_n\|_{H^1}^{2}.
\end{equation*}
Since $v\in L^{\alpha+2}$, we can choose $R\ge1$, depending only on $v$, $\kappa$
and $C_S$, so that $C_S^2\|v\|_{L^{\alpha+2}(|x|\ge R)}^{\alpha}\le\frac\kappa4$.
Collecting the estimates, we obtain
$\kappa\|w_n\|_{H^1}^{2}\le\frac\kappa2\|w_n\|_{H^1}^{2}+e^{2\eps R}B(v)$, that
is, $\|\theta_nv\|_{H^1}^{2}\le\frac2\kappa e^{2\eps R}B(v)$, uniformly in $n$.
As $n\to\infty$, we have $\theta_n v\to e^{\eps|x|}v$ almost everywhere. Since
$(\theta_nv)$ is bounded in $H^1$, a subsequence converges weakly in $H^1$ and,
by the Rellich--Kondrachov theorem, almost everywhere along a further subsequence.
Hence the weak limit is $e^{\eps|x|}v$, and the weak lower semicontinuity of the
norm gives $e^{\eps|x|}v\in H^1$ with
$\|e^{\eps|x|}v\|_{H^1}^{2}\le\frac2\kappa e^{2\eps R}B(v)$. Moreover,
$e^{\eps|x|}\nabla v=\nabla(e^{\eps|x|}v)-\eps\frac{x}{|x|}e^{\eps|x|}v$ almost
everywhere, so that
$e^{2\eps|x|}|\nabla v|^{2}\le2|\nabla(e^{\eps|x|}v)|^{2}+2\eps^{2}e^{2\eps|x|}|v|^{2}$.
Integrating, we obtain
\begin{equation*}
\int e^{2\eps|x|}\big(|\nabla v|^{2}+|v|^{2}\big)\dd
\le(2+2\eps^{2})\,\|e^{\eps|x|}v\|_{H^1}^{2}
\le\frac{4(1+\eps^{2})}{\kappa}\,e^{2\eps R}B(v),
\end{equation*}
which is \eqref{eq:decay}. Finally, $|x|^{2}\le\eps^{-2}e^{2\eps|x|}$, and we
conclude that $|x|v\in L^{2}$.
\end{proof}

\begin{proof}[Proof of Theorem \ref{thm:instability}]
Let $u_c$ be a minimizer of $m_a(c)$. Then $u_c\in V(c)$, $\Ea(u_c)=m_a(c)$, and
$u_c$ is a ground state at the frequency $\omega_c>0$ by Theorem
\ref{thm:level}(iii). By Proposition \ref{prop:decay}, $u_c\in\Sigma$. For $t>1$,
put $v_0^{(t)}:=u_c^{t}$, that is, $v_0^{(t)}(x)=t^{N/2}u_c(tx)$.

We check the required properties. First, $\|v_0^{(t)}\|_{2}=c$ by
\eqref{eq:scalingAB}. Second, $v_0^{(t)}\in\Sigma$, since
$\||x|v_0^{(t)}\|_{2}=t^{-1}\||x|u_c\|_{2}$. Third, the fibering map of $u_c$
attains its strict global maximum at $t_{u_c}=1$, because $\Qa(u_c)=0$. Lemma
\ref{lem:fibering} gives
\begin{equation*}
\Ea(v_0^{(t)})<\Ea(u_c)=m_a(c)
\qquad\text{and}\qquad
\Qa(v_0^{(t)})<0
\qquad\text{for every }t>1 .
\end{equation*}
Fourth, $v_0^{(t)}\to u_c$ in $\Ha$ as $t\downarrow1$. Indeed,
$\|v_0^{(t)}\|_{2}=\|u_c\|_{2}$ and $\|v_0^{(t)}\|_{\dHa}^{2}=t^{2}\|u_c\|_{\dHa}^{2}$,
so $\|v_0^{(t)}\|_{\Ha}\to\|u_c\|_{\Ha}$. Moreover, $v_0^{(t)}\to u_c$ in $L^{2}$
by the continuity of dilations in $L^2$. Since $(v_0^{(t)})$ is bounded in $\Ha$,
every weak limit point equals $u_c$, and $v_0^{(t)}\rightharpoonup u_c$ weakly in
$\Ha$. In a Hilbert space, weak convergence together with convergence of the
norms implies strong convergence.

Given $\eps>0$, we choose $t>1$ so close to $1$ that
$\|v_0^{(t)}-u_c\|_{\Ha}<\eps$. Theorem \ref{thm:dichotomy}(ii) applies to
$v_0^{(t)}$, and the corresponding solution blows up in finite time. Hence
$e^{i\omega_ct}u_c$ is strongly unstable, with destabilizing data on
$S(c)\cap\Sigma$.
\end{proof}

\begin{remark}\rm \label{rmk:instabmeaning}
The instability reflects the minimax nature of the construction. Along its
dilation fibre the normalized ground state sits at the top of a mountain pass, and
the two slopes of the pass are the two sides of the dichotomy in Theorem
\ref{thm:dichotomy}. Spreading the ground state out ($t<1$) produces data below
the level with $\Qa>0$. These data are global, and under \eqref{eq:dynrange} they
scatter by Theorem \ref{thm:dichotomy}(i). Concentrating it ($t>1$) produces data
below the level with $\Qa<0$, which blow up when they have finite variance. The standing
wave $e^{i\omega_ct}u_c$ lies on the boundary between the two regimes, and
$m_a(c)$ is the height of that boundary on the sphere $S(c)$. For the classical
equation, this mechanism goes back to Berestycki and Cazenave
\cite{BerestyckiCazenave1981}; see also \cite{ArdilaCardoso2021} for the
inhomogeneous equation without potential, and \cite{AMK2024} for \eqref{eq:main}.
\end{remark}

%%%%%%%%%%%%%%%%%%%%%%%%%%%%%%%%%%%%%%%%%%%%%%%%%%%%%%%%%%%%%%%%%%%%%%%%%%%%%%
\section{Complements and open problems}\label{sec:remarks}
%%%%%%%%%%%%%%%%%%%%%%%%%%%%%%%%%%%%%%%%%%%%%%%%%%%%%%%%%%%%%%%%%%%%%%%%%%%%%%

\subsection{The mass-critical borderline}\label{ssec:masscritical}

Formula \eqref{eq:levelformula} degenerates as $\alpha$ decreases to
$\frac{4-2b}{N}$, where $\beta\downarrow2$ and $\gamma\uparrow\infty$. The reason is
visible on the fibering map. For $\beta=2$, both terms of $h_u$ carry the same
power of $t$. Hence $h_u$ is a multiple of $t^2$, and there is no mountain pass. What
remains is a threshold in the mass alone.

\begin{proposition}\label{prop:masscritical}
Let $N\ge3$, $a>-\frac{(N-2)^{2}}{4}$, $0<b<2$ and $\alpha=\frac{4-2b}{N}$, so that
$\beta=2$ and $\sigma_c=\alpha$, and set
\begin{equation}\label{eq:cstar}
c^{*}:=\Big(\frac{\alpha+2}{2\,\GN}\Big)^{1/\alpha}.
\end{equation}
Then $\Ea(u)\ge\frac12\big(1-(c/c^{*})^{\alpha}\big)\|u\|_{\dHa}^{2}$ for every
$u\in S(c)$. Consequently $\inf_{S(c)}\Ea=0$ for $c\le c^{*}$ and
$\inf_{S(c)}\Ea=-\infty$ for $c>c^{*}$. Moreover $V(c)=\emptyset$ for $c<c^{*}$,
while for $c\ge c^{*}$ the set $V(c)$ is nonempty and $\Ea\equiv0$ on it, so that
$m_a(c)=0$. If, in addition, Assumption \ref{ass:cauchy} holds for these
parameters, then every solution of \eqref{eq:main} with $\|u_0\|_{2}<c^{*}$ is
global and bounded in $\Ha(\R^{N})$, while for every $c>c^{*}$ there is
$u_0\in S(c)\cap\Sigma$ whose solution blows up in finite time.
\end{proposition}

\begin{proof}
For $\alpha=\frac{4-2b}{N}$ one has
$\sigma_c=\frac{4-2b-(N-2)\alpha}{2}=\frac{N\alpha-(N-2)\alpha}{2}=\alpha$ and
$\beta=\frac{4-2b+2b}{2}=2$, so \eqref{eq:GN} reads $B(u)\le\GN A(u)\|u\|_{2}^{\alpha}$,
and it is sharp and attained by Proposition \ref{prop:GN}. For $u\in S(c)$, we
have
\begin{equation*}
\Ea(u)=\frac{A(u)}{2}-\frac{B(u)}{\alpha+2}
\ \ge\ \frac{A(u)}{2}\Big(1-\frac{2\GN c^{\alpha}}{\alpha+2}\Big)
=\frac{A(u)}{2}\Big(1-\big(c/c^{*}\big)^{\alpha}\Big).
\end{equation*}
If $c\le c^{*}$, the bracket is nonnegative, so $\Ea\ge0$ on $S(c)$. Since
$\Ea(u^{t})=t^2\Ea(u)\to0$ as $t\downarrow0$, the infimum is $0$. If $c>c^{*}$,
then $\GN>\frac{\alpha+2}{2c^{\alpha}}$. Consequently, there is $u\in S(c)$ with
$J(u)>\frac{\alpha+2}{2c^{\alpha}}$, that is, with $\Ea(u)<0$. Then
$\Ea(u^{t})=t^{2}\Ea(u)\to-\infty$ as $t\to\infty$.

Since $\beta=2$, we have $\Qa=2\Ea$. Therefore $\Ea$ vanishes on $V(c)$, and $u\in S(c)$
lies in $V(c)$ exactly when $J(u)=\frac{\alpha+2}{2c^{\alpha}}$. As $J\le\GN$, this
is impossible for $c<c^{*}$. For $c=c^{*}$, any optimizer of mass $c^*$ belongs to
$V(c^*)$. For $c>c^{*}$, we note that $J$ is continuous on the connected set
$S(c)$ and takes the value $\GN$ there. It also takes arbitrarily small values,
since $J(u(\cdot-y))\to0$ as $|y|\to\infty$ for fixed $u$. Indeed, the weight and
the potential both vanish at infinity, so $B(u(\cdot-y))\to0$ while
$A(u(\cdot-y))\to\|\nabla u\|_2^2$. By the intermediate value theorem, $J$ takes
the value $\frac{\alpha+2}{2c^{\alpha}}$ on $S(c)$.

Assume now that Assumption \ref{ass:cauchy} holds. If $\|u_0\|_{2}=c<c^{*}$, then
conservation of mass and energy and the first inequality give
$\|u(t)\|_{\dHa}^{2}\le2\Ea(u_0)\big(1-(c/c^{*})^{\alpha}\big)^{-1}$. Hence the
solution is global by the blow-up alternative. If $c>c^{*}$, take $u\in S(c)$ with
$\Ea(u)<0$ as above. The space $C_c^{\infty}(\R^{N}\setminus\{0\})$ is dense in
$H^1$ for $N\ge3$, and $\Ea$ is continuous on $H^1$. Hence there is $v$ in that
space, close to $u$ in $H^1$, such that $v_0:=(c/\|v\|_2)v$ still has negative
energy. Clearly, $v_0\in S(c)\cap\Sigma$. By Lemma \ref{lem:virial} and the
conservation of energy, we have $V''(t)=8\Qa(u(t))=16\Ea(v_0)<0$. The argument
used in the proof of Theorem \ref{thm:dichotomy}(ii) shows that the solution blows
up in finite time.
\end{proof}

Proposition \ref{prop:masscritical} is consistent with the results of Campos and
Guzm\'an \cite[Theorem~1.2]{CamposGuzman2021}. Let $Q$ be an optimizer of
\eqref{eq:GN} that solves \eqref{eq:ellipticQ}, as given by
\cite[Lemma~3.2]{CamposGuzman2021}. For $\beta=2$, the Pohozaev identities
\cite[Lemma~3.3]{CamposGuzman2021} give $A(Q)=\frac{2}{\alpha}\Mass(Q)$ and
$B(Q)=\frac{\alpha+2}{\alpha}\Mass(Q)$. Equality in \eqref{eq:GN} then yields
$\|Q\|_2^{\alpha}=\frac{\alpha+2}{2\GN}$, that is, $\|Q\|_2=c^*$. The threshold
$\|u_0\|_2<\|Q\|_2$ of \cite[Theorem~1.2(a)]{CamposGuzman2021} is therefore the
threshold $c<c^*$ above. For data of negative energy, finite-time blow-up holds
without any radial or finite-variance assumption, as shown by Boulaaras and the
second author \cite{BS2025}. At the critical mass and for $b=0$, minimal-mass
blow-up solutions were studied by Mukherjee, Nam and Nguyen \cite{MNN2021}.

\subsection{The Hardy endpoint}\label{ssec:hardy}

By Proposition \ref{prop:monotone}, $\GN$ increases as $a$ decreases, and one
may wonder whether it blows up as $a\downarrow\aH$, where \eqref{eq:dHa} ceases to
control $\|\nabla u\|_2^2$. It does not. The reason is that the Hardy functional
\begin{equation*}
\AH(u):=\int|\nabla u|^{2}\dd-\frac{(N-2)^{2}}{4}\int\frac{|u|^{2}}{|x|^{2}}\dd,
\end{equation*}
although not comparable to $\|\nabla u\|_2^2$, still controls a fractional Sobolev
norm once the mass is fixed. This is the case $s=1$, $d=N\ge3$, of a theorem of
Frank \cite[Theorem~1.2]{Frank2009}. Inequality \eqref{eq:frank} below is a
remainder term in the Hardy inequality on the whole space, in a fractional
Sobolev norm and at fixed mass. For every
$t\in(0,1)$, there is $\kappa_{N,t}>0$ such that
\begin{equation}\label{eq:frank}
\AH(u)^{t}\,\|u\|_{2}^{2(1-t)}\ \ge\ \kappa_{N,t}\,\big\|(-\Delta)^{t/2}u\big\|_{2}^{2},
\qquad u\in H^1(\R^N).
\end{equation}
In \cite{Frank2009}, this is stated for $u\in C_c^\infty(\R^N)$. It extends to
$H^1(\R^N)$ by density, since both sides are continuous on $H^1$ by
\eqref{eq:hardy}.

\begin{proposition}\label{prop:hardyendpoint}
Let $N\ge3$, $0<b<2$ and $0<\alpha<\frac{4-2b}{N-2}$. Then $\AH(u)>0$ for every
$u\in H^1(\R^N)\setminus\{0\}$, and
\begin{equation}\label{eq:GNsharp}
\GNs:=\sup\Big\{\frac{B(u)}{\AH(u)^{\beta/2}\|u\|_{2}^{\sigma_c}}:u\in H^1(\R^N)\setminus\{0\}\Big\}<\infty .
\end{equation}
Moreover $\GN<\GNs$ for every $a>\aH$, and $\GN\uparrow\GNs$ as $a\downarrow\aH$.
Consequently, if \eqref{eq:range} holds, then for every $c>0$
\begin{equation*}
m_a(c)\ \downarrow\ K_{\alpha,N,b}\,\GNs^{-\frac{4}{N\alpha+2b-4}}\,c^{-\gamma}>0
\qquad\text{as }a\downarrow\aH .
\end{equation*}
\end{proposition}

\begin{proof}
If $\AH(u)=0$, then \eqref{eq:frank} gives $(-\Delta)^{t/2}u=0$, hence $u=0$. This
proves the first claim.

For the finiteness of $\GNs$, we first establish a weighted inequality involving a
fractional norm. Fix $t\in\big(\frac{\beta}{\alpha+2},1\big)$, which is possible
because $\beta<\alpha+2$. The condition $t>\frac{\beta}{\alpha+2}$ is equivalent to
$(\alpha+2)(N-2t)<2(N-b)$, that is, to $(\alpha+2)\frac{N}{N-b}<\frac{2N}{N-2t}$. As
in the proof of Lemma \ref{lem:weighted}, we can therefore choose $r\in(1,\frac Nb)$
with $2<(\alpha+2)r'\le\frac{2N}{N-2t}$. Moreover $\alpha+2\le\frac{2N}{N-2t}$,
because $t>\frac{\beta}{\alpha+2}>\frac{N\alpha}{2(\alpha+2)}$. Splitting $B(u)$ at
$|x|=1$ and using H\"older's inequality and the embedding
$H^{t}(\R^N)\subset L^{q}(\R^N)$, $2\le q\le\frac{2N}{N-2t}$, we get
$B(u)\le C\big(\|(-\Delta)^{t/2}u\|_{2}^{2}+\|u\|_{2}^{2}\big)^{\frac{\alpha+2}{2}}$.
We apply this to $u^{s}$, for which $B(u^s)=s^{\beta}B(u)$ and
$\|(-\Delta)^{t/2}u^{s}\|_2^{2}=s^{2t}\|(-\Delta)^{t/2}u\|_2^{2}$, and we choose
$s^{2t}=\|u\|_2^{2}/\|(-\Delta)^{t/2}u\|_2^{2}$. We get
\begin{equation*}
B(u)\le C'\,\big\|(-\Delta)^{t/2}u\big\|_{2}^{\frac{\beta}{t}}\,\|u\|_{2}^{\alpha+2-\frac{\beta}{t}} .
\end{equation*}
Inserting \eqref{eq:frank}, we obtain
\begin{equation*}
B(u)\le C'\kappa_{N,t}^{-\frac{\beta}{2t}}\,\AH(u)^{\frac\beta2}\,
\|u\|_{2}^{\frac{(1-t)\beta}{t}+\alpha+2-\frac{\beta}{t}}
=C'\kappa_{N,t}^{-\frac{\beta}{2t}}\,\AH(u)^{\frac\beta2}\,\|u\|_{2}^{\sigma_c},
\end{equation*}
since $\frac{(1-t)\beta}{t}-\frac\beta t=-\beta$ and $\alpha+2-\beta=\sigma_c$.
Thus $\GNs<\infty$.

For $a>\aH$ and $u\ne0$, we have
$A(u)=\AH(u)+(a-\aH)\int|u|^2|x|^{-2}\ge\AH(u)>0$, hence $J(u)$ is at most the
quotient in \eqref{eq:GNsharp}, and $\GN\le\GNs$. The inequality is strict, because
$\GN<\mathcal{C}_{a'}\le\GNs$ for any $a'\in(\aH,a)$, by Proposition
\ref{prop:monotone}. On the other hand, for fixed $u\ne0$, $A(u)\downarrow\AH(u)>0$
as $a\downarrow\aH$, so $J(u)$ converges to the quotient in \eqref{eq:GNsharp}.
Hence $\liminf_{a\downarrow\aH}\GN$ is at least that quotient for every $u$, that
is, at least $\GNs$. Together with the monotonicity this shows $\GN\uparrow\GNs$.
The statement on $m_a(c)$ follows from \eqref{eq:levelformula}.
\end{proof}

Proposition \ref{prop:hardyendpoint} corrects a natural but misleading intuition.
The threshold $m_a(c)$ does not shrink to zero as the potential becomes critically
attractive. It stays bounded below by a positive constant, and the limiting level
$K_{\alpha,N,b}\GNs^{-\frac{4}{N\alpha+2b-4}}c^{-\gamma}$ is a natural candidate for
the threshold of the equation with $a=\aH$, whose Cauchy theory was developed by
Suzuki \cite{Suzuki2016} (see also \cite{Suzuki2012}). For
$a=\aH$, the energy space is the completion of $C_c^\infty(\R^N\setminus\{0\})$ for
the norm $\big(\AH(u)+\|u\|_2^2\big)^{1/2}$, which is strictly larger than
$H^1(\R^N)$.

\subsection{Open problems}\label{ssec:open}

We end with questions that the present analysis does not settle. None of the
results above depends on them. We group them into variational and dynamical
questions.

\smallskip
\noindent\emph{Variational questions.}
\begin{enumerate}
\item[(i)] \emph{Uniqueness and nondegeneracy of ground states.} For $a=b=0$, the
positive solution of \eqref{eq:ellipticQ} is unique up to translation, by the
moving-plane method and the uniqueness theorem of Kwong \cite{Kwong1989}. For
$a=0$ and $b>0$, the uniqueness of the positive radial solution of
\eqref{eq:ellipticQ} follows from Yanagida's criterion \cite{Yanagida1991} for
$N\ge3$ (see \cite{GenoudStuart2008}). For $b=0$ and $\aH\le a<0$, Mukherjee, Nam and
Nguyen \cite{MNN2021} proved, for $N\ge3$ and every energy-subcritical power, that
the positive radial solution is unique. They also proved that every optimizer of
the corresponding Hardy--Gagliardo--Nirenberg inequality is a multiple of a dilate
of it. An alternative proof of this uniqueness, based on a shooting argument,
was recently given by Yang, Zeng and Zhang \cite{YZZ2026}. For $b=0$ and $a>0$, there is no optimizer (Remark \ref{rmk:attained}).
When $a\ne0$ and $b\ne0$, we are not aware of any uniqueness result. By Theorem
\ref{thm:level}, the ground states at a given frequency are exactly the
Gagliardo--Nirenberg optimizers of the corresponding mass lying on $V(c)$.
Uniqueness of one is therefore equivalent to uniqueness of the other. The
question is the following. \emph{For $a\ne0$ and $0<b<2$, is the ground state of
\eqref{eq:ellipticQ} unique up to phase? If $Q$ is a positive ground state, is it
nondegenerate, in the sense that the operator
$\Lc+1-(\alpha+1)|x|^{-b}Q^{\alpha}$ has a trivial kernel in $\Ha(\R^N)$?}

\item[(ii)] \emph{Symmetry of the optimizers.} For $\aH<a\le0$, Schwarz
symmetrization does not increase $\|\nabla u\|_2$ and does not decrease
$\int|u|^2|x|^{-2}$ or $B(u)$, by the P\'olya--Szeg\H{o} inequality and the
Hardy--Littlewood rearrangement inequality. Hence $J$ does not decrease,
and optimizers may be chosen radial and nonincreasing. For $a>0$, the potential
term moves in the wrong direction under rearrangement. \emph{For $a>0$, are all
optimizers of \eqref{eq:GN} radial, up to phase? Or can symmetry breaking occur
for large $a$?}

\item[(iii)] \emph{The Hardy endpoint.} \emph{Is the supremum $\GNs$ in
\eqref{eq:GNsharp} attained, in $H^1(\R^N)$ or in the larger energy space of the
operator $-\Delta-\frac{(N-2)^2}{4}|x|^{-2}$?} For $b=0$, the critical coupling
is the main case treated by Mukherjee, Nam and Nguyen \cite{MNN2021}, who obtained
ground states and a sharp Hardy--Gagliardo--Nirenberg inequality in the larger
energy space. A positive answer
would give
normalized solutions for this critical operator. \emph{Is the limiting level in
Proposition \ref{prop:hardyendpoint} the sharp threshold for the equation with
$a=\aH$?}
\end{enumerate}

\smallskip
\noindent\emph{Dynamical questions.}
\begin{enumerate}
\item[(iv)] \emph{Scattering in the whole intercritical range.} The variational
results hold for all $N\ge3$ and all $\alpha$ in \eqref{eq:range}. So do the
global existence statement and the blow-up part of Theorem \ref{thm:dichotomy}.
The scattering part is stated here only under \eqref{eq:dynrange}, where we rely
on \cite{HLW2024}. For $N=3$ and $a>0$, Theorem \ref{thm:equiv} and
\cite[Theorem~2.2]{BGS2023} give scattering below $m_a(c)$ when $\Qa(u_0)>0$, for
all $0<b<2$ and all $\alpha$ in \eqref{eq:range} (see the comparison in
Subsection \ref{ssec:novelties}). The question is the following. \emph{Let $N\ge3$,
$a>\aH$, $0<b<2$ and let $\alpha$ satisfy \eqref{eq:range}. Does every solution
with $\Ea(u_0)<m_a(c)$ and $\Qa(u_0)>0$ scatter in $\Ha(\R^N)$?} The cases left
open by \cite{HLW2024,BGS2023} include $N\ge4$, and $N=3$ with $a<0$ and either
$\alpha\ne2$ or $b\ge1$. The main obstacles are the lack of unweighted dispersive
estimates for $a<0$, the restriction to radial data of the dispersive estimates
of \cite{Zheng2018} for $N\ge4$, and the restrictions of the Strichartz-space
theory in Table~\ref{tab_LWP}. We expect a positive answer, at least for $a\ge0$
and for radial data (compare \cite{Zheng2018,LuMiaoMurphy2018}), but we do not
prove it.

\item[(v)] \emph{Finite-time blow-up without finite variance.} Our blow-up
statement assumes $u_0\in\Sigma$. For radial data of positive energy, finite-time
blow-up under the conditions of Theorem \ref{thm:dichotomy}(ii) follows from
\cite[Theorem~1.3(b)]{CamposGuzman2021} and Theorem \ref{thm:equiv}. For general
data, An, Jang and Kim \cite{AJK2023} obtained criteria for global existence and
blow-up of solutions which are not necessarily radial or of finite variance. In
the mass-critical case, finite-time blow-up for negative energy holds without
radial or finite-variance assumptions \cite{BS2025}. For non-radial data of
negative energy, non-global solutions were obtained by Bai and the second author
\cite{BaiSaanouni2025}. For $a>0$ and $N=3$, Campos, Cardoso and Farah
\cite{CCF2023} proved the blow-up of the critical Lebesgue norm, including for
some non-radial data with $\Ea(u_0)\le0$. For $a=0$, blow-up criteria for
non-radial data were given by Dinh and Keraani \cite{DinhKeraani2021}. The
question is the following. \emph{Let $u_0\in\Ha(\R^N)$ be neither radial nor of finite variance,
with $\Ea(u_0)<m_a(c)$ and $\Qa(u_0)<0$. Does the solution blow up in finite
time?} The decay of the weight $|x|^{-b}$ makes localized virial arguments
available without symmetry assumptions, as in \cite{BS2025,DinhKeraani2021}. For the classical equation
$a=b=0$, this question is open in general.

\item[(vi)] \emph{Dynamics at the level.} We do not study data with
$\Ea(u_0)=m_a(c)$. For the classical cubic equation, the solutions on the
threshold were classified by Duyckaerts and Roudenko \cite{DuyckaertsRoudenko2010},
and for $b=0$ and $\aH<a<0$ by Yang, Zeng and Zhang \cite{YZZ2026}.
Theorem \ref{thm:instability} and Remark \ref{rmk:instabmeaning} show that, along
its dilation fibre, the standing wave is unstable towards blow-up on one side and
towards global existence on the other. This is consistent with the expected
picture of a one-dimensional unstable direction. \emph{Classify the solutions
with $\Ea(u_0)=m_a(c)$.} Such a classification would require the nondegeneracy
asked for in (i).

\item[(vii)] \emph{Dependence of the dynamics on the coupling.} By Proposition
\ref{prop:monotone}, the level $m_a(c)$ increases with $a$. \emph{Is the set of
data satisfying \eqref{eq:below} and $\Qa(u_0)>0$ monotone in $a$?} As explained
in Remark \ref{rmk:monotoneset}, this does not follow from the monotonicity of the
level, because the energy also increases with $a$.
\end{enumerate}

\bigskip

\section*{Declarations}
\noindent\textbf{Funding.} This research received no funding.

\noindent\textbf{Conflict of interest.} The author declares no conflict of interest.

\noindent\textbf{Data availability.} Data sharing is not applicable to this article, as no datasets were generated or analysed.

\end{document}